\documentclass{article}

\usepackage{amsfonts}
\usepackage{amsmath}
\usepackage{amssymb}
\usepackage{amsthm}

\usepackage{authblk}

\usepackage{multicol}
\usepackage[margin=3cm]{geometry} 
\usepackage{soul} 

\usepackage[pagebackref, pdftex]{hyperref}

\renewcommand*{\backref}[1]{}
\renewcommand*{\backrefalt}[4]{
  \ifcase #1
  [No citations.]
  \or [#2]
  \else [#2]
  \fi }

\usepackage{marginnote} 
\long\def\@savemarbox#1#2{\global\setbox#1\vtop{\hsize\marginparwidth\@parboxrestore\tiny\raggedright #2}}
\usepackage{graphicx}
\usepackage{import}

\usepackage{tikz}
\usetikzlibrary{calc, arrows, decorations.markings, decorations.pathmorphing, positioning, decorations.pathreplacing}

\AtBeginDocument{%
   \def\MR#1{}
}

\newcommand{\To}{\longrightarrow}
\newcommand{\C}{\mathbb{C}}
\newcommand{\D}{\mathcal{D}}
\newcommand{\Disc}{\mathbb{D}}
\newcommand{\F}{\mathcal{F}}
\newcommand{\HH}{\mathcal{H}}
\newcommand{\hyp}{\mathbb{H}}
\newcommand{\R}{\mathbb{R}}
\newcommand{\RP}{\mathbb{RP}}
\renewcommand{\S}{\mathcal{S}}
\newcommand{\U}{\mathbb{U}}
\newcommand{\Z}{\mathbb{Z}}
\newcommand{\ZZ}{\mathcal{Z}}
\DeclareMathOperator{\Fr}{Fr}
\DeclareMathOperator{\Gr}{Gr}
\DeclareMathOperator{\Hor}{Hor}
\let\Im\relax
\DeclareMathOperator{\Im}{Im}
\let\Re\relax
\DeclareMathOperator{\Re}{Re}
\DeclareMathOperator{\Mat}{Mat}
\DeclareMathOperator{\Spin}{Spin}
\DeclareMathOperator{\Tr}{Tr}

\numberwithin{equation}{section}

\newtheorem{theorem}{Theorem}

\newtheorem{lem}[equation]{Lemma}

\newtheorem{prop}[equation]{Proposition}

\newtheorem{defn}[equation]{Definition}

\newcommand{\refsec}[1]{Section~\ref{Sec:#1}}
\newcommand{\refdef}[1]{Definition~\ref{Def:#1}}
\newcommand{\reffig}[1]{Figure~\ref{Fig:#1}}

\newcommand{\refeqn}[1]{\eqref{Eqn:#1}}
\newcommand{\reflem}[1]{Lemma~\ref{Lem:#1}}
\newcommand{\refprop}[1]{Proposition~\ref{Prop:#1}}
\newcommand{\refthm}[1]{Theorem~\ref{Thm:#1}}

\begin{document}

\title{Spinors and horospheres} 

\author{Daniel V. Mathews} 
\affil{School of Mathematics, Monash University \\
\texttt{Daniel.Mathews@monash.edu}}

\date{}

\maketitle

\begin{abstract}
We give an explicit bijective correspondence between between nonzero pairs of complex numbers, which we regard as spinors or spin vectors, and horospheres in 3-dimensional hyperbolic space decorated with certain spinorial directions. This correspondence builds upon work of Penrose--Rindler 
and Penner.
We show that the natural bilinear form on spin vectors describes a certain complex-valued distance between spin-decorated horospheres, generalising Penner's lambda lengths to 3 dimensions. 

From this, we derive several applications. We show that the complex lambda lengths in a hyperbolic ideal tetrahedron satisfy a Ptolemy equation. 
We also obtain correspondences between certain spaces of hyperbolic ideal polygons 
and certain Grassmannian spaces, under which lambda lengths correspond to Pl\"{u}cker coordinates, illuminating the connection between Grassmannians, hyperbolic polygons, and type A cluster algebras. 
\end{abstract}

\section{Introduction}

Penrose and Rindler in \cite{Penrose_Rindler84} describe various aspects of relativity theory in terms of spinorial objects. The foundation of their spinorial description of spacetime is a construction associating objects of Minkowski space $\R^{1,3}$ to vectors $\kappa = (\xi, \eta) \in \C^2$, which they call \emph{spin vectors} and which, for present purposes, we simply call \emph{spinors}. To nonzero spinors, Penrose and Rindler associate a \emph{null flag}, which is a point $p$ on the future light cone $L^+$, together with a 2-plane tangent to $L^+$ containing $p$, which behaves in a spinorial way. See \reffig{1} (left).

On the other hand, Penner in \cite{Penner87} introduced a ``decorated Teichm\"{u}ller theory" which has since been highly developed (see e.g. \cite{Penner12_book}). A basic construction of this theory relates points on the future light cone $L^+$ in Minkowski space of one lower dimension $\R^{1,2}$, to \emph{horocycles} in the hyperbolic plane $\hyp^2$, which sits in Minkowski space as the hyperboloid model, consisting of all points 1 unit in the future of the origin. See \reffig{1} (right).

\begin{figure}[h]
\begin{center}
\begin{tikzpicture}
  \draw[black] (3.75,1.5) ellipse (2cm and 0.3cm);
  \fill[white] (2.75,0.5)--(4.75,0.5)--(4.75,0.72)--(2.75,0.72);
  \draw[black] (2.75,0.5)--(3.25,0);
  \draw[black] (2.75,0.5)--(1.75,1.5);
  \draw[black] (4.25,0)--(4.75,0.5);
  \draw[black] (4.75,0.5)--(5.75,1.5);
  \draw[black] (3.25,0)--(3.75,-0.5)--(4.25,0.0);
  \draw[red] (3.75,-0.5)--(4,0);
  \draw[red] (4,0)--(4.1875,0.375);
  \draw[red] (4.1375,0.275)--(4.475,0.95)--(4.675,0.75)--(4.275,0.55);
  \node[black] at (1.5,1.5){$L^+$};
  \fill[red] (4.475,0.95) circle (0.055cm);
\begin{scope}[xshift=10cm,yshift=-1cm,scale=0.8]
  \draw[black] (-4,4)--(0,0)--(4,4);
  \draw[blue, dashed, thick] plot[variable=\t,samples=1000,domain=-75.5:75.5] ({tan(\t)},{sec(\t)});
  \draw[black] (0,4) ellipse (4cm and 0.4cm);
  \draw[blue, dotted, thick] (-0.2,3.7) .. controls (-1,0.25) .. (1.8,4.27);
  \draw[blue] (0,4) ellipse (3.85cm and 0.3cm);
  \draw[red] (0,0)--(2,3);
  \fill[red] (2,3) circle (0.1cm); 
  \node[black] at (-3.5,3){$L^+$};
  \node[blue] at (-0.75,2.5){$H$};
  \node[blue] at (-2.25,3){$\hyp^2$};
\end{scope}
\end{tikzpicture}
\caption{Left: null flag corresponding to a spinor. Right: point on light cone (red), with corresponding horosphere $H$ in $\hyp^2$ (blue).}
\label{Fig:1}
\end{center}
\end{figure}
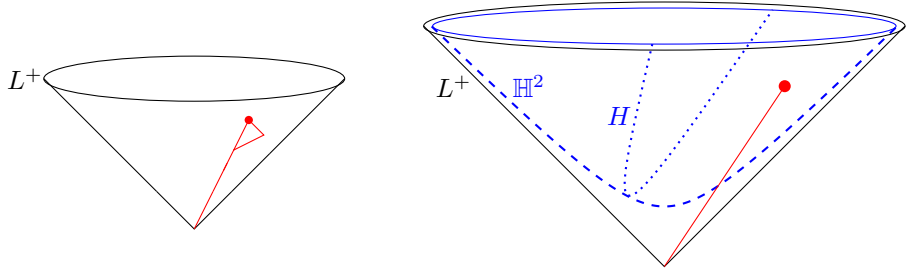

In this paper we observe that these two constructions can be combined and generalised, yielding the following theorem.
\begin{theorem}
\label{Thm:main_thm_1}
There is a smooth, bijective, $SL(2,\C)$-equivariant correspondence between nonzero spinors and spin-decorated horospheres in hyperbolic 3-space $\hyp^3$.
\end{theorem}
A \emph{decoration} on a horosphere is a tangent parallel oriented line field, i.e. a choice of direction along the horosphere at each point which is invariant under parallel translation. Such decorations exist since horospheres in $\hyp^3$ are isometric to the Euclidean plane. A \emph{spin decoration} is, roughly, a ``spin lift" of such a decoration, where rotating the direction by $2\pi$ is not the identity, but rotation by $4\pi$ is; we define them precisely in \refsec{spin_decorations}. 
See \reffig{2} (left).

The correspondence of \refthm{main_thm_1} is expressed very simply in the upper half space model $\U^3$ of $\hyp^3$. 
(Because we use several models of hyperbolic geometry, we denote them distinctly: $\hyp^n, \U^n, \Disc^n$ denote the hyperboloid, upper half plane/space, and conformal ball/disc models of hyperbolic $n$-space respectively. When our considerations are independent of model we simply use $\hyp^n$.)
Regarding the sphere at infinity $\partial \U^3$ of $\U^3$ as $\C \cup \{\infty\}$ in the usual way, the horosphere $H$ corresponding to $(\xi, \eta)$ has centre at $\xi/\eta$. If $\xi/\eta = \infty$ then $H$ is a horizontal plane in $\U^3$ at height $|\xi|^2$, and its decoration points in the direction $i\xi^2$. If $\xi/\eta \in \C$ then $H$ is a Euclidean sphere in $\U^3$ of Euclidean diameter $|\eta|^{-2}$, and its decoration at its highest point (``north pole") points in the direction $i \eta^{-2}$. See \reffig{2} (right). We prove this in \refprop{horospheres_explicitly}.

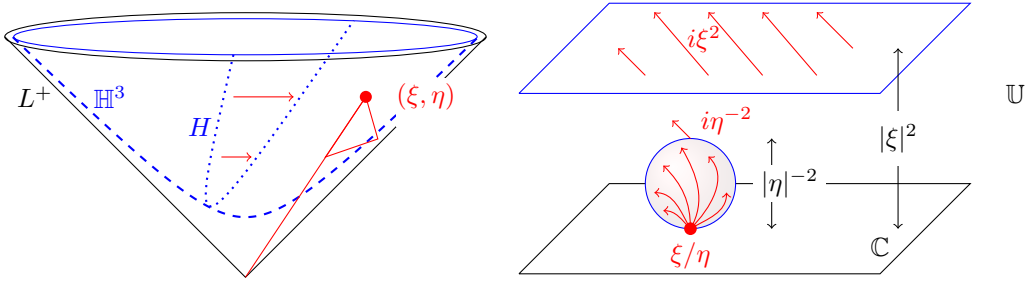
\begin{figure}[h]
\begin{center}
\begin{tabular}{cc}
\begin{tikzpicture}[scale=0.8]
  \draw[black] (-4,4)--(0,0)--(4,4);
  \draw[blue, dashed, thick] plot[variable=\t,samples=1000,domain=-75.5:75.5] ({tan(\t)},{sec(\t)});
  \draw[black] (0,4) ellipse (4cm and 0.4cm);
  \draw[blue, dotted, thick] (-0.2,3.7) .. controls (-1,0.25) .. (1.8,4.27);
  \draw[blue] (0,4) ellipse (3.85cm and 0.3cm);
  \draw[red] (0,0)--(2,3);
  \fill[red] (2,3) circle (0.1cm); 
  \node[black] at (-3.5,3){$L^+$};
	\fill[white](2.5,2.5)--(3.5,2.5)--(3.5,3.3)--(2.5,3.3)--cycle;
  \node[red] at (3,3){$(\xi,\eta)$};
  \draw[red] (2,3)--(2.2,2.3)--(1.33,2)--(2,3);
  \node[blue] at (-0.75,2.5){$H$};
  \node[blue] at (-2.25,3){$\hyp^3$};
  \draw[red, ->] (-0.2,3)--(0.8,3);
  \draw[red, ->] (-0.4,2)--(0.1,2);
\end{tikzpicture}
&
\begin{tikzpicture}[scale=1.2]
    \draw[black] (-2,-0.5)--(2,-0.5)--(3,0.5)--(-1,0.5)--(-2,-0.5);
    \fill[white] (-0.1,0.5) circle (0.5cm);
    \shade[ball color = red!40, opacity = 0.1] (-0.1,0.5) circle (0.5cm);
    \draw[blue] (-0.1,0.5) circle (0.5cm);
		\fill[red] (-0.1,0) circle (0.07cm);
    \draw[->, red] (-0.1,0) to[out=135,in=0] (-0.4,0.2);
    \draw[->, red] (-0.1,0) to[out=120,in=0] (-0.5,0.4);
    \draw[->, red] (-0.1,0) to[out=90,in=-45] (-0.4,0.7);
    \draw[->, red] (-0.1,0) to[out=60,in=-60] (-0.2,0.9);
    \draw[->, red] (-0.1,0) to[out=45,in=-45] (0.1,0.8);
    \draw[->, red] (-0.1,0) to[out=30,in=-90] (0.3,0.4);
		\draw[red, ->] (-0.1,1)--(-0.3,1.2);
		\node[red] at (0.3,1.2) {$i \eta^{-2}$};
		\node[red] at (-0.1,-0.3) {$\xi/\eta$};
		\draw[<->] (0.8,0)--(0.8,1);
		\fill[white] (0.6,0.3)--(1.4,0.3)--(1.4,0.7)--(0.6,0.7)--cycle;
		\node[black] at (1,0.5) {$|\eta|^{-2}$};
    \draw[blue] (-2,1.5)--(2,1.5)--(3,2.5)--(-1,2.5)--(-2,1.5);
    \begin{scope}[xshift=0.5cm]
		\draw[red,->] (-1.1,1.7)--(-1.4,2);
    \draw[red,->] (-0.4,1.7)--(-1,2.4);
    \draw[red,->] (0.2,1.7)--(-0.4,2.4);
    \draw[red,->] (0.8,1.7)--(0.2,2.4);
    \draw[red,->] (1.2,2)--(0.8,2.4);
		\node[red] at (-0.45,2.1) {$i \xi^2$};
		\end{scope}
		\draw[<->] (2.2,0)--(2.2,2);
		\fill[white] (1.8,0.7)--(2.6,0.7)--(2.6,1.3)--(1.8,1.3)--cycle;
		\node[black] at (2.2,1) {$|\xi|^2$};
		\node[black] at (3.5,1.5) {$\U$};
		\node[black] at (2,-0.2) {$\C$};
		\end{tikzpicture}
\end{tabular}
\caption{Left: Null flag corresponding to $(\xi, \eta)$, and corresponding horosphere. Right: decorated horospheres as they appear in the upper half space model $\U$.}
\label{Fig:2}
\end{center}
\end{figure}

As indicated, both sides of this correspondence have $SL(2,\C)$-actions. The action on $\C^2$ is by the standard action of matrices on vectors. The action on horospheres is via the action of $PSL(2,\C)$ as orientation-preserving isometries of $\hyp^3$; this action lifts to $SL(2,\C)$ to preserve spin structures.

Penrose and Rindler in \cite{Penrose_Rindler84} define a complex-valued antisymmetric bilinear form $\{ \cdot, \cdot \}$ on spinors, given by the $2 \times 2$ determinant; it can also be regarded as the standard complex symplectic form on $\C^2$. On the other hand, given two horospheres $H_1, H_2$, we may measure the signed distance $\rho$ from one to the other along their mutually perpendicular geodesic. As detailed in \refsec{spin_decorations}, we may also measure an angle $\theta$ between spin decorations, which is well defined modulo $4\pi$. We regard $d = \rho + i \theta$ as a \emph{complex distance} between the two horospheres. See \reffig{3}. 
\begin{figure}[h]
\def\svgwidth{0.38\columnwidth}
\begin{center}
%% Creator: Inkscape 1.2 (dc2aedaf03, 2022-05-15), www.inkscape.org
%% Accompanies image file 'complex_lambda_lengths.pdf' (pdf, eps, ps)
\begingroup%
  \makeatletter%
  \providecommand\color[2][]{%
    \errmessage{(Inkscape) Color is used for the text in Inkscape, but the package 'color.sty' is not loaded}%
    \renewcommand\color[2][]{}%
  }%
  \providecommand\transparent[1]{%
    \errmessage{(Inkscape) Transparency is used (non-zero) for the text in Inkscape, but the package 'transparent.sty' is not loaded}%
    \renewcommand\transparent[1]{}%
  }%
  \providecommand\rotatebox[2]{#2}%
  \newcommand*\fsize{\dimexpr\f@size pt\relax}%
  \newcommand*\lineheight[1]{\fontsize{\fsize}{#1\fsize}\selectfont}%
  \ifx\svgwidth\undefined%
    \setlength{\unitlength}{247.70801514bp}%
    \ifx\svgscale\undefined%
      \relax%
    \else%
      \setlength{\unitlength}{\unitlength * \real{\svgscale}}%
    \fi%
  \else%
    \setlength{\unitlength}{\svgwidth}%
  \fi%
  \global\let\svgwidth\undefined%
  \global\let\svgscale\undefined%
  \makeatother%
  \begin{picture}(1,0.64212724)%
    \lineheight{1}%
    \setlength\tabcolsep{0pt}%
    \put(0,0){\includegraphics[width=\unitlength,page=1]{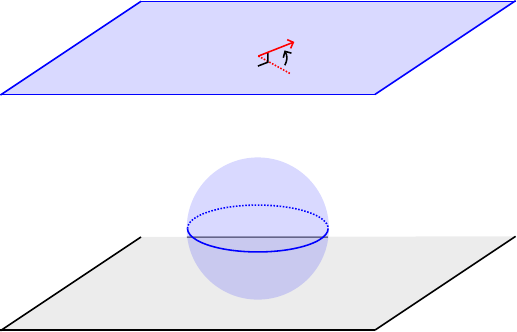}}%
    \put(0.77703916,0.58512719){\color[rgb]{0,0,0}\makebox(0,0)[lt]{\lineheight{1.25}\smash{\begin{tabular}[t]{l}$H_1$\end{tabular}}}}%
    \put(0,0){\includegraphics[width=\unitlength,page=2]{complex_lambda_lengths.pdf}}%
    \put(0.45375586,0.10699279){\color[rgb]{0,0,0}\makebox(0,0)[lt]{\lineheight{1.25}\smash{\begin{tabular}[t]{l}$H_2$\end{tabular}}}}%
    \put(0,0){\includegraphics[width=\unitlength,page=3]{complex_lambda_lengths.pdf}}%
    \put(0.39200427,0.4027793){\color[rgb]{0,0,0}\makebox(0,0)[lt]{\lineheight{1.25}\smash{\begin{tabular}[t]{l}$\rho$\end{tabular}}}}%
    \put(0.57579002,0.51343701){\color[rgb]{0,0,0}\makebox(0,0)[lt]{\lineheight{1.25}\smash{\begin{tabular}[t]{l}$\theta$\end{tabular}}}}%
    \put(0,0){\includegraphics[width=\unitlength,page=4]{complex_lambda_lengths.pdf}}%
  \end{picture}%
\endgroup%
 
\caption{Complex distance between horospheres.}
\label{Fig:3}
\end{center}
\end{figure}

The correspondence of \refthm{main_thm_1} extends further to relate these structures, as follows.
\begin{theorem}
\label{Thm:main_thm_2}
Given two spinors $\kappa_1, \kappa_2$, suppose the corresponding spin-decorated horospheres have complex distance $d$. Then
\[
\{\kappa_1, \kappa_2\} = \exp \left( \frac{d}{2} \right).
\]
\end{theorem}
In the 2-dimensional context, the distance between horocycles is just a real number $d$, and the quantity $\lambda = \exp(d/2)$ is known as a \emph{lambda length} \cite{Penner87}. Thus, the standard bilinear form on spinors computes lambda lengths between corresponding spin-decorated horospheres.

{\flushleft \textbf{Example.} }
Take 
$\kappa_1 = (1,0)$ and $\kappa_2 = (0,1)$.
In the upper half space model, 
$\kappa_1$
corresponds to a horosphere centred at $1/0 = \infty$, so appears as a horizontal plane, at height $1$ and with decoration in the direction $i$. Similarly, the horosphere of 
$\kappa_2$
has centre $0/1 = 0$, appearing as a sphere centred at $0$, with Euclidean diameter $1$, and direction at its highest point in the direction $i$. These two horospheres are tangent at the point $(0,0,1)$, where their decorations align. It turns out their spin directions also align, so $\rho = \theta = 0$ and hence $\lambda = 1 = 
\{\kappa_1, \kappa_2\}$.

After multiplying 
$\kappa_1$
by a complex number $re^{i\phi}$, with $r>0$ and $\phi$ real, its corresponding horosphere remains centred at $\infty$, but the horizontal plane moves to height $r^2$, translated upwards from its original position by hyperbolic distance $2 \log r$, and its decoration is rotated by $2 \phi$. From 
$\kappa_1$ to $\kappa_2$
we then have $\rho = 2 \log r$ and $\theta = 2 \phi$, so $\lambda = \exp(\frac{1}{2}(2 \log r + 2 \phi i)) = re^{i\phi} = 
\{\kappa_1, \kappa_2\}$.

{\flushleft \textbf{Approach.} }
This paper proceeds through the constructions and proofs in a self-contained manner; we need to adapt and extend the constructions of Penrose--Rindler and Penner for our purposes. To give a rough overview, from a spinor $\kappa = (\xi, \eta)$, following Penrose--Rindler, we obtain points $(T,X,Y,Z)$ on the positive light cone $L^+$ in $\R^{1,3}$ via
\begin{equation}
\label{Eqn:basic_correspondence}
\begin{pmatrix} \xi \\ \eta \end{pmatrix}
\begin{pmatrix} \overline{\xi} & \overline{\eta} \end{pmatrix}
=
\frac{1}{2} \begin{pmatrix} T+Z & X+iY \\ X-iY & T-Z \end{pmatrix}
\end{equation}
which uses the correspondence between Hermitian $2 \times 2$ matrices and Minkowski space $\R^{1,3}$ given by Pauli matrices. The flag of $\kappa$ has flagpole along the ray of $(T,X,Y,Z)$ determined by \refeqn{basic_correspondence}, and 2-plane defined by the derivative of this map $\kappa \mapsto (T,X,Y,Z)$ in a certain direction depending on $\kappa$; we show this is a variation on the Penrose--Rindler construction. The correspondence between spinors and flags is $SL(2,\C)$-equivariant, where $SL(2,\C)$ acts on flags via its action on $\R^{1,3}$ in standard fashion as $SO(1,3)^+$. In \refsec{spin_vectors_to_Minkowski} we describe these constructions precisely, along with explicit calculations, and details of $SL(2,\C)$-equivariance, which we have not seen proved elsewhere in the literature.

From a point $p$ on the figure light cone $L^+$, Penner's construction in \cite{Penner87} is to consider the affine plane in Minkowski space consisting of all $x$ satisfying 
\[
\langle x, p \rangle = 1,
\]
where $\langle \cdot, \cdot \rangle$ is the standard Lorentzian metric. This affine plane intersects the hyperboloid model of hyperbolic space $\hyp^3$ in a horosphere $H$. This construction works in any dimension, and in $\R^{1,3}$, we obtain an affine 3-plane which intersects the hyperboloid model of $\hyp^3$ in a 2-dimensional horosphere. We show that when $p$ arises from a spinor $\kappa$, this affine 3-plane contains an affine 2-plane parallel to the flag 2-plane of $\kappa$, and this affine 2-plane intersects the horosphere in a parallel oriented line field on $H$, yielding the picture of \reffig{2} (left). Again, all constructions are $SL(2,\C)$-equivariant. Precise details are given in \refsec{Minkowski_horospheres}. 

Finally, all these constructions lift, in appropriate sense, to spin double covers, and remain $SL(2,\C)$-equivariant, as we detail in \refsec{spin_decorations}.

Strictly speaking, \refthm{main_thm_2} could be obtained more quickly, by directly defining the decorated horospheres associated to spinors in the manner of \refprop{horospheres_explicitly}. However, our approach places these results in the original context of Minkowski and spin geometry, as developed by Penrose and Rindler, illustrating various connections and applications.

{\flushleft \textbf{Hyperbolic geometry applications.} }
These theorems have several applications and in this paper we consider some of them. In \refsec{hyperbolic_applications} we consider some applications to hyperbolic geometry. 

Consider an ideal hyperbolic tetrahedron, i.e. with all vertices on $\partial \hyp^3$. Number the vertices $0,1,2,3$. 
We consider a spin decoration on the tetrahedron, consisting of a spin-decorated horosphere $H_i$ at each ideal vertex $i$. There is then a complex lambda length $\lambda_{ij}$ from $H_i$ to $H_j$. Each $\lambda_{ij}$ measures, in a certain sense, the distance between horospheres along each edge, along with the angle between them. See \reffig{TetrLabels}. 

\begin{theorem}
\label{Thm:main_thm_3}
The complex lambda lengths $\lambda_{ij}$ in a spin-decorated ideal tetrahedron satisfy
\[
\lambda_{01} \lambda_{23} + \lambda_{03} \lambda_{12} = \lambda_{02} \lambda_{13}.
\]
\end{theorem}
This equation is similar to Ptolemy's theorem relating the lengths of sides and diagonals of a cyclic quadrilaterals in classical Euclidean geometry, hence we call it a \emph{Ptolemy equation}. Penner in \cite{Penner87} proved a corresponding Ptolemy equation in 2 dimensions: when an ideal quadrilateral in $\hyp^2$ has its vertices decorated with horocycles, the (real) lambda lengths of the edges and diagonals satisfy the same equation. \refthm{main_thm_3} is a 3-dimensional generalisation showing that Ptolemy's equation still holds, once we take horospheres to have spin decorations, and take lambda lengths to be complex. Roughly, 2-dimensional hyperbolic geometry corresponds to the case when spinors are \emph{real}, i.e. lie in $\R^2$.

\begin{figure}[h]
\begin{center}
    \begin{tikzpicture}[scale=2.5]
    \draw (-1,0)--(1.5,0.5);
		    \fill[white] (0.75,0.35) circle (0.1 cm);
    \draw (0,1.5)--(-1,0)--(1,0)--(0,1.5)--(1.5,0.5)--(1,0);
    \draw[blue] (-0.83,0.1) circle (0.2);
    \draw[blue] (0.85,0.12) circle (0.2);
    \draw[blue] (0,1.3) circle (0.2);
    \draw[blue] (1.3,0.5) circle (0.2);
    \shade[ball color = blue!40, opacity = 0.1] (-0.83,0.1) circle (0.2cm);
    \shade[ball color = blue!40, opacity = 0.1] (0.85,0.12) circle (0.2cm);
    \shade[ball color = blue!40, opacity = 0.1] (0,1.3) circle (0.2cm);
    \shade[ball color = blue!40, opacity = 0.1] (1.3,0.5) circle (0.2cm);
		\draw [red, thick, ->] (1.05,0.05) -- (0.95,0.07);
		\draw [red, thick, ->] (0.79,0.31) -- (0.95,0.35);
		\draw [red, thick, ->] (0.69,0) -- (0.72,0.12);
		\draw [red, thick, ->] (0.18,1.38) -- (0.24,1.23);
		\draw [red, thick, ->] (0.18,1.22) -- (0.11,1.04);
		\draw [red, thick, ->] (-0.18,1.22) -- (-0.25,1.4);
		\draw [red, thick, ->] (1.22,0.68) -- (1.35,0.7);
		\draw [red, thick, ->] (1.3,0.3) -- (1.25, 0.45);
		\draw [red, thick, ->] (1.12,0.42) -- (1.1, 0.6);
		\draw [red, thick, ->] (-0.63,0.08) -- (-0.65,-0.04);
		\draw [red, thick, ->] (-0.8,0.3) -- (-0.7,0.25);
		\draw [red, thick, ->] (-0.66,0) -- (-0.8,-0.1);
		\draw [black!50!green, ultra thick, ->] (-0.5,-0.1) to [out=0, in=180] (0.5,0.1);
		\draw [black!50!green] (0,-0.2) node {$\lambda_{12}$};
		\draw [black!50!green, ultra thick, ->] (-0.4,1.1) to [out=240, in=60] (-0.6,0.4);
		\draw [black!50!green] (-0.7,0.75) node {$\lambda_{01}$};
		\draw [black!50!green, ultra thick, ->] (0.22,1) to [out=-60, in=120] (0.78,0.5);
		\draw [black!50!green] (0.4,0.65) node {$\lambda_{02}$};
		\draw [black!50!green, ultra thick, ->] (1.15,0.05) to [out=45, in=250] (1.18,0.27);
		\draw [black!50!green] (1.365,0.16) node {$\lambda_{23}$};
		\draw [black!50!green, ultra thick, ->] (0.35,1.17) to [out=-33, in=147] (1.15,0.85);
		\draw [black!50!green] (0.85,1.11) node {$\lambda_{03}$};
		\draw (0,1.7) node {$0$};
		\draw (-1.2,-0.2) node {$1$};
		\draw (1.2,-0.2) node {$2$};
		\draw (1.7,0.6) node {$3$};
		
\end{tikzpicture}
  \caption{Decorated horospheres and complex lambda lengths along the edges of an ideal tetrahedron.}
  \label{Fig:TetrLabels}
\end{center}
\end{figure}
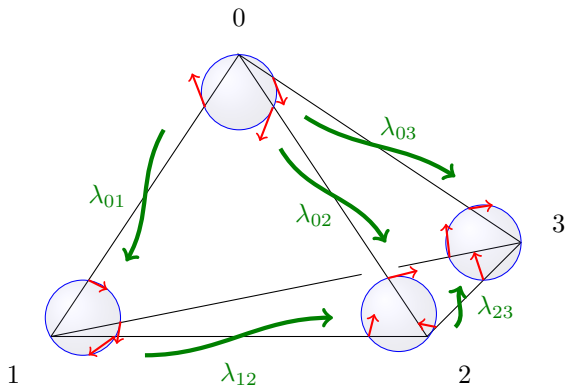

With the previous theorems in hand, the proof of \refthm{main_thm_3} is not difficult. Indeed, the four spin-decorated horospheres correspond to four spinors in $\C^2$, which can be arranged into a $2 \times 4$ matrix. The Ptolemy equation is then just the Pl\"{u}cker relation between $2 \times 2$ determinants of a $2 \times 4$ matrix. Alternatively, it can be seen as the relation
\[
\varepsilon_{AB} \varepsilon_{CD} + \varepsilon_{BC} \varepsilon_{AD} + \varepsilon_{CA} \varepsilon_{BC} = 0
\]
satisfied by the spinor $\varepsilon$, as in \cite{Penrose_Rindler84} (e.g. eq. 2.5.21).

A related result is given in \refprop{shape_parameters}, where we show that the shape parameters of an ideal tetrahedron can be recovered from these six lambda lengths.

Truncations of ideal tetrahedra along horospheres arise naturally, for instance, in complete hyperbolic structures on 3-manifolds. In a forthcoming paper with Purcell \cite{Mathews_Purcell_Ptolemy} we show how Ptolemy equations can be used to describe hyperbolic structures on 3-manifolds, giving a directly hyperbolic-geometric version of the Ptolemy equations described by Garoufalidis--Thurston--Zickert \cite{GTZ15} and enhanced Ptolemy variety of Zickert \cite{Zickert16}, in turn based on work of Fock--Goncharov \cite{Fock_Goncharov06}.

Even in 2 dimensions, spinors provide a useful way to analyse the geometry of horocycles; we take the spinors to have real coordinates. In forthcoming work with Zymaris we apply this to circle packing theory and generalise a classical theorem of Descartes \cite{Mathews_Zymaris}.

Indeed, when $\xi$ and $\eta$ are both \emph{integers} that are relatively prime, then the horocycles obtained in the upper half plane model of $\hyp^2$ are the \emph{Ford circles}, with their delightful relationships to Farey fractions, Diophantine approximation and continued fractions \cite{Ford_1938}.

{\flushleft \textbf{Cluster algebra applications.} }
In \refsec{cluster_applications} we consider some applications to cluster algebras. We refer to \cite{Williams14} for an introduction to basic notions of cluster algebras.

We have already mentioned how 4 spinors arising from a spin-decorated ideal tetrahedron $\Delta$ can be arranged into a $2 \times 4$ matrix. Considering those 4 spinors up to the common action of $SL(2,\C)$ corresponds to considering such a $\Delta$ up to isometry. And considering appropriate $2 \times 4$ matrices up to a left action by $2 \times 2$ matrices is a standard description of a \emph{Grassmannian}. Thus, the correspondences of the above theorems yield a relationship between hyperbolic geometry and Grassmannians.

In \cite[sec. 12.2]{Fomin_Zelevinsky03}, Fomin--Zelevinsky described a geometric realisation of cluster algebras of type $A_n$ in terms of Grassmannians. Clusters in this case are in bijection with triangulations of an $(n+3)$-gon; two clusters are joined by an edge in the exchange graph if and only if the triangulations are related by a flip \cite{Chapoton_Fomin_Zelevinsky_02, Fomin_Zelevinsky_03_Y-systems}. Fomin--Zelevinsky showed that the cluster algebra is realised by $X(n+3)$, the affine cone over the Grassmannian $\Gr(2,n+3)$ in particular, the cluster algebra there denoted $\mathcal{A}_\circ$ in type $A_n$ is isomorphic to the $\Z$-form of the coordinate ring $\C[X(n+3)]$, with the cluster variables mapping to Pl\"{u}cker coordinates. This has since been generalised in various ways, for instance to other Grassmannians \cite{Scott06} and partial flag varieties \cite{Geiss_Leclerc_Schroer_08}.

On the other hand, work of Fock--Goncharov \cite{Fock_Goncharov06}, Gekhtman--Shapiro--Vainshtein \cite{Gekhtman_Shapiro_Vainshtein_05}, Fomin--Shapiro--Thurston \cite{Fomin_Shapiro_Thurston08} and Fomin--Thurston \cite{Fomin_Thurston18} provides geometric realisations of cluster algebras arising from surfaces, in terms of the decorated Teichm\"{u}ller space $\widetilde{T}(n+3)$ introduced by Penner \cite{Penner87}, using lambda lengths. In the particular case of an $(n+3)$-gon, the lambda lengths of the diagonals provide cluster variables for the cluster algebra of type $A_n$ \cite[examples 8.10, 16.1]{Fomin_Thurston18}. Fock--Goncharov in \cite{Fock_Goncharov06} also give numerous results relating Teichm\"{u}ller spaces to higher algebraic structures, but as far as they relate to hyperbolic geometry and the results of this paper, they are in dimension 2. As mentioned earlier, the Ptolemy equations of Garoufalidis--Thurston--Zickert \cite{GTZ15}, which use variables provided by Fock--Goncharov, are given a 3-dimensional hyperbolic-geometric interpretation by our complex lambda lengths in forthcoming work with Purcell \cite{Mathews_Purcell_Ptolemy} .

In any case, there is thus a well-understood isomorphism between the cluster algebras arising from the affine cone $X(n+3)$ over the Grassmannian $\Gr(2,n+3)$, and from the decorated Teichm\"{u}ler space $\widetilde{T}(n+3)$. In \cite[remark 16.2]{Fomin_Thurston18} Fomin--Thurston note a connection between the underlying spaces.

The results of this paper further illuminate the situation, by giving a direct identification of the spaces underlying these cluster algebras, and extending them to 3 dimensions. The correspondence between spinors and spin-decorated horospheres naturally yields identifications of certain decorated Teichm\"{u}ller spaces, and certain Grassmannian spaces, as follows.

\begin{theorem}
\label{Thm:T_Gr}
\label{Thm:2d_T_Gr}
\label{Thm:3d_T_Gr}
Let $d \geq 3$. The correspondence of \refthm{main_thm_1} yields the following identifications.
\begin{enumerate}
\item
The decorated Teichm\"{u}ller space $\widetilde{T}(d)$ of ideal $d$-gons in $\hyp^2$ is identified with the affine cone $X^+ (n)$ on the positive Grassmannian $\Gr^+ (2,d)$. 
\item
The decorated Teichm\"{u}ller space of ideal skew $n$-gons in $\hyp^3$ is identified with the affine cone $X^* (d)$ on the subvariety of the complex Grassmannian $\Gr (2,d)$ where all Pl\"{u}cker coordinates are nonzero. 
\end{enumerate}
Under each identification, lambda lengths correspond to Pl\"{u}cker coordinates.
\end{theorem}
In \refsec{cluster_applications} we define all notions precisely and prove some properties about them, including these theorems. 

The rough idea is simply that a collection of $n$ spinors $\kappa_1, \ldots, \kappa_n$ describes the $n$ ideal vertices of an ideal $n$-gon (in 2 dimensions), or an ideal skew $n$-gon (in 3 dimensions); but on the other hand, we may place the $\kappa_i$ as the $n$ columns of a $2 \times n$ matrix. The appropriate decorated Teichm\"{u}ller space is then given by the certain (spin) isometry classes of such ideal (skew) $n$-gons, which is an orbit space of a $d$-tuple of spinors $(\kappa_1, \ldots, \kappa_d)$. The corresponding set of orbits of $2 \times n$ matrices gives the affine cone on the appropriate Grassmannian.

It is not difficult to vary the conditions on $d$-gons or Grassmannians, and find identifications between diverse versions of decorated Teichm\"{u}ller spaces, and corresponding diverse Grassmannian spaces.

The appearance of the \emph{positive} Grassmannian here corresponds to the fact, proved in \refprop{Td_totally_positive_spinors}, that a sequence of horocycles with lambda lengths that are \emph{positive}, in an appropriate sense, corresponds to the the centres of the horocycles being in order around $\partial \hyp^2$. Positivity of determinants and Grassmannians and their relationship to ordered or convex objects arises in a similar way in the physics of scattering amplitudes, see e.g. \cite{ABCGPT16}.

{\flushleft \textbf{Other related work.} }
In addition to the literature already mentioned, notions related to the present paper have been discussed in several different contexts.

In \cite{FKST_23}, Felikson--Karpenkov--Serhiyenko--Tumarkin also consider an extension of lambda lengths to a 3-dimensional context, associating to an unordered pair of horospheres a nonnegative-real-valued lambda length, which is the modulus of our lambda length. They also show that their lambda lengths are given by (moduli of) determinants, using fractions over the Eisenstein integers, and give several relations satisfied by their lambda lengths. Given five horospheres, the lambda lengths from one horosphere to the four others satisfy a quartic equation \cite[theorem 4.12]{FKST_23}, which can be regarded as a corollary of the Soddy--Gosset theorem in 3 dimensions; and the 10 lambda lengths involved satisfy a degree-10 polynomial \cite[theorem 4.17]{FKST_23}. This degree-10 polynomial can be found as the necessarily zero determinant of a $5 \times 5$ matrix with $(m,n)$ entry $|z_m - z_n|^2$, where $z_n \in \C$ is the centre of the $n$th horosphere; these entries can be regarded as their lambda lengths squared. In a similar way, our Ptolemy equation, up to sign ambiguities, can be recovered as the zero determinant of the $4 \times 4$ matrix with $(m,n)$ entry $(z_m - z_n)^2$, which is the lambda length squared between spin-decorated horospheres with spinors $(z_n,1)$.\footnote{I am indebted to an anonymous referee for pointing out this reference and this argument.} 

In \cite{Springborn_17}, Springborn gives a real, 2-dimensional version of Theorems 1--3, associating to each nonzero pair $(p,q) \in \R^2$ the horocycle in the upper half space model of $\hyp^2$ centred at $p/q$ with Euclidean diameter $1/q^2$, showing this association is equivariant and that the signed hyperbolic distance between an (unordered) pair of horospheres is given by twice the logarithm of the associated determinant. Springborn considers in detail the Ford circles and Farey tessellation associated to relatively prime $(p,q)$, relating them to Diophantine approximation and Markov’s theorem. 

In \cite{McShane_Sergiescu}, McShane--Sergiescu consider real, 2-dimensional lambda lengths between Ford circles and show they are given by the absolute value of the associated determinant. They use these and related ideas to give a proof of Fermat's sum of two squares theorem. In \cite{McShane_Eisenstein} McShane uses similar ideas in the 2-dimensional context, extended to the Eisenstein integers, to prove further number-theoretic results.

Gelfand--MacPherson in 1982 \cite{Gelfand_MacPherson_82} considered various aspects of polytopes in Grassmannians. As part of this work, they noted that the configuration space $C_n^m$ of $(m+n)$ points in $\R P^{m-1}$ is diffeomorphic to the quotient of $\tilde{G}_n^m$, 
the subspace of the Grassmannian of $n$-planes in $\R^{m+n}$ consisting of planes intersecting each coordinate $m$-plane only at the origin, by the group $H \subset PGL(m+n)$ of diagonal matrices. Their $C_n^2$ thus consists of configurations of $n$ points in $\R P^1 \cong \partial \hyp^2$, which is similar to our space of ideal $d$-gons, except that the ideal vertices of our $d$-gons are required to be in order around $\partial \hyp^2$. They also consider ``enhanced" projective configurations, which form a $2^{m+n-1}$-sheeted covering space of $C_n^m$. These correspond to the possible choices of planar spin decorations (\refdef{planar_spin_decoration}) on a horocycle centred at each point, up to an overall choice of sign.
In \cite{MOST_14}, Morier-Genoud--Ovsienko--Schwartz--Tabachnikov consider this correspondence further, and the relationships between the objects involved and linear difference equations and frieze patterns. In particular, in \cite[sec. 8.4]{MOST_14} they consider relations to Teichm\"{u}ller theory, including cross-ratios and a notion of alternating perimeter length, resulting in similar expressions as we obtain for tetrahedron shape parameters in \refprop{shape_parameters}.

{\flushleft \textbf{Acknowledgments.} }
The author thanks Varsha for assistance in the preparation of figures.
He also thanks the anonymous referees for their thoughtful comments and suggestions, and for bringing to his attention several related works.
He is supported by Australian Research Council grant DP210103136.

\section{Spinors to Hermitian matrices and Minkowski space}
\label{Sec:spin_vectors_to_Minkowski}

For us \emph{spinors} are just elements of $\C^2$, which we regard as a complex symplectic vector space, with complex symplectic form denoted 
\[
\{ \cdot, \cdot \} = d\xi \wedge d\eta
\]
following \cite{Penrose_Rindler84}. We denote spinors by $\kappa = (\xi, \eta)$ or similar.
 Given $\kappa = (\xi, \eta)$ and $\kappa' = (\xi', \eta')$ then
\[
\{ \kappa, \kappa' \} = \xi \eta' - \eta \xi' = \det \begin{pmatrix} \xi & \xi' \\ \eta & \eta' \end{pmatrix}.
\]
We write $\det(\kappa, \kappa')$ for the above determinant. We denote by
$\C_*^2$ 
the set of nonzero spinors.

For the purposes of linear algebra, we regard $\kappa$ as a column vector and write $\kappa^T$ for the corresponding row vector. The adjoint $\kappa^* = \overline{\kappa}^T$ is then a row vector.

We map spinors into the set $\HH$ of Hermitian $2 \times 2$ matrices, or equivalently into Minkowski space $\R^{1,3}$. We take $\R^{1,3}$ to have coordinates $(T,X,Y,Z)$ and metric $dT^2 - dX^2 - dY^2 - dZ^2$, denoted $\langle \cdot, \cdot \rangle$. We observe $\HH$ and $\R^{1,3}$ are isomorphic 4-dimensional real vector spaces and we identify them in a standard way (perhaps the constant is slightly unorthodox)
\[
(T,X,Y,Z) \leftrightarrow \frac{1}{2} \begin{bmatrix} T+Z & X+iY \\ X-iY & T-Z \end{bmatrix}.
\]
The right hand expression is $\frac{1}{2} \left( T + X \sigma_X + Y \sigma_Y + Z \sigma_Z \right)$, where the $\sigma_\bullet$ are the Pauli matrices.
If a point $x = (T,X,Y,Z) \in \R^{1,3}$ corresponds to $S \in \HH$ then we observe $\Tr S = T$ and $4 \det S = \langle x, x \rangle$. The light cone $L = \{ x \in \R^{1,3} \; \mid \; \langle x,x \rangle = 0 \}$ corresponds to $S$ with determinant zero, and the future light cone $L^+ = \{ x \in L \; \mid \; T > 0 \}$ corresponds to $S$ satisfying $\det S = 0$ and $\Tr S > 0$. We define the \emph{celestial sphere} $\S^+$ to be the intersection of $L^+$ with the 3-plane $T=1$.

\begin{defn}[\cite{Penrose_Rindler84}]
\label{Def:phi_1}
The map $\phi_1$ from $\C^2$ to $\HH \cong \R^{1,3}$ is defined by $\phi_1 (\kappa) = \kappa \kappa^*$.
\end{defn}
In other words,
\[
\phi_1 (\kappa) = \begin{pmatrix} \xi \\ \eta \end{pmatrix} \begin{pmatrix} \overline{\xi} & \overline{\eta} \end{pmatrix} = \begin{pmatrix} |\xi|^2 & \xi \overline{\eta} \\ \overline{\xi} \eta & |\eta|^2 \end{pmatrix}.
\]
We observe that the image of $\phi_1$ has determinant zero, and its diagonal entries are $|\xi|^2, |\eta|^2$, so that its trace is non-negative. Indeed it is not difficult to show $\phi_1( \C_*^2 ) = L^+$. Thus $\phi_1$ maps a 4-(real)-dimensional domain onto a 3-dimensional image. The fibres are circles; it is not difficult to show that $\phi_1 (\kappa) = \phi_1 (\kappa')$ iff $\kappa = e^{i\theta} \kappa'$ for some real $\theta$. Indeed, on each 3-sphere in $\C^2$ given by $\kappa$ with $|\xi|^2 + |\eta|^2$ fixed at some constant $c>0$, $\phi_1$ restricts to the Hopf fibration onto the 2-sphere in $L^+$ given by $T = c$. Thus $\phi_1$ is the cone on the Hopf fibration.

In order not to lose information, we extend $\phi_1$ to a map including tangent data. Given a tangent vector $\nu$ in the real tangent space $T_\kappa \C_*^2$, we write $D_\kappa \phi_1(\nu)$ for the derivative of $\phi_1$ at $\kappa$ in the direction $\nu$. Since, for real $t$, 
\[
\phi_1 \left( \kappa + t \nu \right) = \left( \kappa + t \nu \right) \left( \kappa + t \nu \right)^* = \kappa \kappa^* + \left( \kappa \nu^* + \nu \kappa^* \right) t + \nu \nu^* t^2
\]
we have
\begin{equation}
\label{Eqn:derivative_phi1}
D_\kappa \phi_1 (\nu) = \left. \frac{d}{dt} \phi_1 \left( \kappa + t \nu \right) \right|_{t=0} = \kappa \nu^* + \nu \kappa^*.
\end{equation}
In the abstract index notation of \cite{Penrose_Rindler84}, this directional derivative is $\kappa^A \overline{\nu}^{A'} + \nu^A \overline{\kappa}^{A'}$.
At each point $\kappa$ we will build a flag structure using the derivative in a certain direction $\ZZ(\kappa)$. 

\begin{defn}
\label{Def:Z}
The function $\ZZ \colon \C^2 \To \C^2$ is given by $\ZZ(\xi, \eta) = (i \overline{\eta}, -i \overline{\xi})$. In other words,
\[
\ZZ \kappa = J \overline{\kappa} \quad \text{where} \quad
J = \begin{bmatrix} 0 & i \\ -i & 0 \end{bmatrix}.
\]
\end{defn}

Let us attempt to motivate this definition. Penrose--Rindler use a spinor $\tau^A$ forming a spin frame, or standard symplectic basis, with $\kappa^A$, i.e. so that $\{ \kappa, \tau \} = \kappa_A \tau^A = 1$. They then form a 2-plane defined by the bivector $K^a \wedge L^b$ where $K^a = \kappa^A \overline{\kappa}^{A'}$ and $L^a = \kappa^A \overline{\tau}^{A'} + \tau^A \overline{\kappa}^{A'}$. These two vectors are our $\phi_1 (\kappa)$ and $D_\kappa \phi_1 (\tau)$. But the same oriented 2-plane is obtained using any positive multiple of such $\tau$, so we could equally fix $\kappa_A \tau^A$ simply to be positive real. Choosing $\tau$ to make $\kappa_A \tau^A$ negative real, or positive/negative imaginary, works also for our purposes. Our choice of $\ZZ$ ensures $\{ \kappa, \ZZ(\kappa) \} = -i(|\xi|^2 + |\eta|^2)$ is negative imaginary. Though somewhat arbitrary, this works well for our purposes.

Another perspective on $\ZZ$ is obtained by identifying $(\xi, \eta) \in \C^2$ with the quaternion $\xi + \eta j$. Then $\ZZ \kappa = -k \kappa$. On the $S^3$ centred at the origin in $\C^2$ through $\kappa$, the tangent space at $\kappa$ has basis $i \kappa, j \kappa, k \kappa$. In the $i \kappa$ direction lies the fibre $e^{i\theta} \kappa$, and $\phi_1$ is constant; $\ZZ \kappa$ is another tangent vector to this $S^3$. 

In any case, \refeqn{derivative_phi1} and \refdef{Z} immediately yield
\begin{equation}
\label{Eqn:DkappaZkappa}
D_\kappa \phi_1 ( \ZZ \kappa ) = \kappa \kappa^T J + J \overline{\kappa} \kappa^*.
\end{equation}

We now define the type of flag structure we need.
\begin{defn}
An \emph{oriented flag} of signature $(d_1, \ldots, d_k)$ in a real vector space $V$ is an increasing sequence of subspaces
\[
\{0\} = V_0 \subset V_1 \subset \cdots \subset V_k
\]
where $\dim V_i = d_i$, and for $i=1, \ldots, k$, the quotient $V_{i}/V_{i-1}$ is endowed with an orientation.
\end{defn}

\begin{defn}
A \emph{pointed oriented null flag}, or just \emph{flag}, consists of a point $p \in L^+$ and an oriented flag $\{0\} \subset V_1 \subset V_2$ in $\HH \cong \R^{1,3}$ of signature $(1,2)$, such that
\begin{enumerate}
\item $V_1 = \R p$ and the orientation on $V_1$ is towards the future (i.e. from $0$ towards $p$),
\item $V_2$ is a tangent plane to $L^+$.
\end{enumerate}
The set of flags is denoted $\F$.
\end{defn}
Thus $p$ is 
on a
flagpole $\R p$, which runs towards the future along the light cone; and the flag plane $V_2$ is a tangent plane to the light cone, with its relative orientation equivalent to choosing the half-plane to one side of $\R p$ or the other. Note that $V_2$ contains no timelike vectors, and $\R p$ generates the unique 1-dimensional lightlike subspace of $V_2$. The tangent space to $L^+$ at $p$ is defined by the equation $\langle x, p \rangle = 0$, i.e. is the (Minkowski-)orthogonal complement $p^\perp$. Thus $\R p \subset V_2 \subset p^\perp$.

Given linearly independent $p \in L^+$ and $v \in T_p L^+$, we denote by $[[p,v]]$ the flag given by $p$, the line $\R p$ oriented from the origin towards $p$, the plane $V_2$ spanned by $p$ and $v$, and the orientation on $V_2/\R p$ induced by $v$. We observe that two flags so given $[[p,v]]$, $[[p',v']]$ are equal if and only if $p=p'$ 
and there exist real $a,b,c$ such that $ap+bv+cv' = 0$, where $b,c$ (which are necessarily nonzero) have opposite sign.

Note that $\F$ is diffeomorphic to $UTS^2 \times \R$, where $UTS^2$ is the unit tangent bundle of $S^2$: a point of $S^2$ describes a future-oriented ray in $L^+$, a unit tangent vector there describes a relatively oriented 2-plane, and the $\R$ factor fixes $p$ along the ray. Since $UTS^2 \cong \RP^3$ we also have $\F \cong \RP^3 \times \R$

Our version of Penrose--Rindler null flags can now be defined as the following map, upgrading $\phi_1$.
\begin{defn}
The map $\Phi_1$ maps nonzero spinors to (pointed oriented null) flags via
\[
\Phi_1 \colon \C_*^2 \To \F, \quad
\Phi_1 (\kappa) = [[ \phi_1 (\kappa), D_\kappa \phi_1 (\ZZ \kappa) ]].
\]
\end{defn}
Thus the point $\phi_1 (\kappa)$ yields the flagpole, and the derivative of $\phi_1$ in the $\ZZ \kappa$ direction yields the relatively oriented flag plane. We verify that $D_\kappa \phi_1 (\ZZ \kappa)$ is (real-)linearly independent from $\phi_1 (\kappa)$ using \refeqn{DkappaZkappa}: if $a \kappa \kappa^* + b \left( \kappa \kappa^T J + J \overline{\kappa} \kappa^* \right) = 0$ for some real $a,b$ then $\kappa \left( a \kappa^* + b \kappa^T J \right) = \left( J \overline{\kappa} \right) \left( - b \kappa^* \right)$; both sides of this equation being the product of a $2 \times 1$ and $1 \times 2$ matrix, the corresponding matrices must be proportional, say $\kappa = c J \overline{\kappa}$ for some real $c$; in components then $\xi = c i \overline{\eta}$ and $\eta = -c i \overline{\xi}$, so $\xi = -c^2 \xi$ and $\eta = -c^2 \eta$, so that $\xi = \eta = 0$, a contradiction.

\begin{lem}
\label{Lem:when_flags_equal}
For two spinors $\kappa, \nu \in \C_*^2$, the following are equivalent:
\begin{enumerate}
\item
$\{ \kappa, \nu \}$ is negative imaginary (just like $\{ \kappa, \ZZ \kappa \}$);
\item
$\nu = a \kappa + b \ZZ \kappa$ where $a$ is complex and $b$ is real positive;
\item
$[[\phi_1 (\kappa), D_\kappa \phi_1 (\nu) ]] = [[ \phi_1 (\kappa), D_\kappa \phi_1 (\ZZ \kappa) ]] = \Phi_1 (\kappa).$
\end{enumerate}
\end{lem}

\begin{proof}
If $\{ \kappa, \nu \}$ is negative imaginary then $\{ \kappa, b \ZZ \kappa \} = \{ \kappa, \nu \}$ for some positive $b$, and any two vectors yielding the same value for $\{ \kappa, \cdot \}$ differ by a complex multiple of $\kappa$. This shows (i) implies (ii), and the converse is clear. 

If $\nu = a \kappa + b \ZZ \kappa$ then by real-linearity of the derivative, $D_\kappa \phi_1 (\nu) = D_\kappa \phi_1 ( a \kappa ) + b D_\kappa \phi_1 ( \ZZ \kappa )$. The derivative of $\phi_1$ in the $\kappa$ direction is proportional to $\phi_1 (\kappa)$, and the derivative  in the $i \kappa$ direction is zero (pointing along a fibre of $\phi_1$). Thus the derivatives in the $\nu$ and $\ZZ \kappa$ directions span the same plane when taken together with $\phi_1 (\kappa)$; indeed, as $b>0$, the same relatively oriented plane. In fact, this condition is equivalent to spanning the same relatively oriented plane.
\end{proof}

The spaces $\C^2$, $\HH \cong \R^{1,3}$ and $\F$ all have natural $SL(2,\C)$ actions; in all cases we denote the action of $A \in SL(2,\C)$ by a dot. An $A \in SL(2,\C)$ acts on $\C^2$ by the defining representation, $A.\kappa = A \kappa$, yielding a symplectomorphism:
\[
\{A \kappa, A \kappa'\} = \det (A \kappa, A \kappa') 
= \det A (\kappa, \kappa') 
= \det (\kappa, \kappa') = \{ \kappa, \kappa' \}
\]
since $\det A = 1$. 
The same $A$ acts on $S \in \HH$ by $A.S = ASA^*$, which in $\R^{1,3}$ yields in the standard way the linear maps of $SO(1,3)^+$, i.e. those which preserve the Minkowski metric and space and time orientation. The action on $\R^{1,3}$ induces orientation-preserving actions on $L^+$ and planes in $\R^{1,3}$, yielding an action on $\F$, so that $A.[[p,v]] = [[A.p,A.v]]$. Essentially by definition $\phi_1$ is equivariant with respect to these actions,
\[
\phi_1 (A \kappa) = A \kappa \kappa^* A^* = A \phi_1 (\kappa) A^* = A.\phi_1(\kappa),
\]
and we have an equivariance property on its derivatives
\[
A. D_\kappa \phi_1 (\nu) = D_{A \kappa} \phi_1 (A \nu)
\]
since $A.(\kappa \nu^* + \nu \kappa^*) = A \kappa \nu^* A^* + A \nu \kappa^* A^* = (A \kappa) (A \nu)^* + (A \nu) (A \kappa)^*$. We now show the equivariance property extends to $\Phi_1$; we have not seen a proof of this in the existing literature.
\begin{lem}
The map $\Phi_1$ is equivariant with respect to the $SL(2,\C)$ actions on $\C_*^2$ and $\F$.
\end{lem}

\begin{proof}
We have $\Phi_1 (\kappa) = [[ \phi_1 (\kappa), D_\kappa \phi_1 (\ZZ \kappa) ]]$ so 
\[
A.\Phi_1 (\kappa) = [[ A.\phi_1 (\kappa), A.D_\kappa \phi_1 (\ZZ \kappa) ]]
= [[ \phi_1 (A \kappa), D_{A \kappa} \phi_1 ( A (\ZZ \kappa) ) ]],
\]
by equivariance of $\phi_1$ and its derivative. Now as $A$ is symplectic, $\{ A \kappa, A (\ZZ \kappa) \} = \{ \kappa, \ZZ \kappa \}$, which is negative imaginary, so by \reflem{when_flags_equal} then $[[ \phi_1 (A \kappa), D_{A \kappa} \phi_1 (A (\ZZ \kappa) ) ]] = \Phi_1 (A \kappa)$.
\end{proof}

It is possible to express explicitly the linear dependence implied by the equality of the flags $\Phi_1 (A \kappa) = [[ \phi_1 (A \kappa), D_{A \kappa} \phi_1 ( \ZZ(A \kappa) ) ]]$ and $A.\Phi_1 (\kappa) = [[ \phi_1 (A \kappa), D_{A \kappa} \phi_1 (A ( \ZZ \kappa) ) ]]$: a direct computation verifies the (perhaps surprising) identity
\[
\left( \kappa^T J A^* A \kappa + \kappa^* A^* A J \overline{\kappa} \right)
\phi_1 (A \kappa)
+
\left( \kappa^* \kappa \right)
\left[ D_{A \kappa} \phi_1 \left( \ZZ \left( A \kappa \right) \right) \right]
-
\left( \kappa^* A^* A \kappa \right)
\left[ D_{A \kappa} \phi_1 \left( A \left( \ZZ \kappa \right) \right) \right]
=
0.
\]

We now compute $\Phi_1$ explicitly in coordinates.
\begin{lem}
\label{Lem:Phi1_calculation}
Let $\kappa = (\xi, \eta) = (a+bi, c+di)$. Then in $\R^{1,3}$ we have
\[
\phi_1 (\kappa) = \left( a^2 + b^2 + c^2 + d^2, 2(ac+bd), 2(bc-ad), a^2 + b^2 - c^2 - d^2 \right),
\]
and 
\[
D_\kappa \phi_1 (\ZZ \kappa) = 
\left( 0, 2(cd-ab), a^2 - b^2 + c^2 - d^2, 2(ad+bc) \right).
\]
\end{lem}
The fact that $D_\kappa \phi_1 (\ZZ \kappa)$ has zero $T$-coordinate follows from $\ZZ \kappa$ being tangent to the $S^3$ centred at the origin through $\kappa$, which maps under $\phi_1$ to the $S^2$ given by the intersection of $L^+$ with a plane at constant $T$.

\begin{proof}
This is a straightforward computation using \refdef{phi_1}
\[
\phi_1 (\kappa) = \begin{bmatrix} \xi \overline{\xi} & \xi \overline{\eta} \\ \overline{\xi} \eta & \eta \overline{\eta} \end{bmatrix}
= \begin{bmatrix} a^2 + b^2 & (ac+bd)+(bc-ad)i \\ (ac+bd)-(bc-ad)i & c^2 + d^2 \end{bmatrix}
\]
and, via \refeqn{DkappaZkappa},
\[
\D_\kappa \phi_1 (\ZZ \kappa)
= \kappa \kappa^T J + J \overline{\kappa} \kappa^*
= \begin{bmatrix} i \left( \overline{\xi \eta} - \xi \eta \right) & i \left( \xi^2 + \overline{\eta}^2 \right) \\ i \left( \overline{\xi}^2 + \eta^2 \right) & i \left( \xi \eta - \overline{\xi \eta} \right) \end{bmatrix}.
\]
\end{proof}
We denote by $\hyp^3$ the hyperboloid model of hyperbolic 3-space
\[
\hyp^3 = \left\{ x = (T,X,Y,Z) \in \R^{1,3} \; \mid \; \langle x,x \rangle = 1, \; T>0 \right\}
\]
and by $\partial \hyp^3$ the boundary at infinity of $\hyp^3$. So $\partial \hyp^3 \cong S^2$ and $\partial \hyp^3$ is naturally bijective with the celestial sphere $\S^+$.

Indeed, projectivising $L^+$ yields 
the boundary at infinity $\partial \hyp^3 \cong S^2$  and under this projectivisation, 2-planes tangent to $L^+$ containing a ray of $L^+$ correspond bijectively with tangent lines at points of $\partial \hyp^3$. Moreover, relatively oriented planes containing a ray of $L^+$ correspond bijectively with tangent directions at points of $\partial \hyp^3$. 

The orientation-preserving isometry group $SO(1,3)^+$ of $\hyp^3$ acts transitively on the future light cone $L^+$, and indeed acts transitively on the tangent directions at points of $\partial \hyp^3$. Further, if we take an oriented flag consisting of a future-oriented line $R$ of $L^+$ and a relatively oriented 2-plane $\pi$ tangent to $L^+$, then there is an element of $SO(1,3)^+$ fixing $R$ (and its orientation) and $\pi$ (and its relative orientation), which sends any point on the ray to any other. Such an element is given by a hyperbolic translation along any geodesic with an endpoint at infinity corresponding to $R$. In other words, $SL(2,\C)$ acts transitively on $\F$, and the action factors through $PSL(2,\C) \cong SO(1,3)^+$. 

Taking $\kappa = (e^{i\theta},0)$, we have $\phi_1 (\kappa) = (1,0,0,1)$, which we denote $p_0$, and by \reflem{Phi1_calculation}, $\Phi_1 (\kappa)$ is the flag with basepoint $p_0$ and 2-plane spanned by $p_0$ and $(0,-\sin 2\theta, \cos 2\theta,0)$. 
Thus as we multiply $\kappa$ by $e^{i\theta}$ to move through a fibre of $\phi_1$, the flag $\Phi_1 (\kappa)$ rotates about a fixed pointed flagpole twice as fast.
It follows that $\Phi_1$ takes the value of each such flag exactly twice. 

Using the equivariance of $\Phi_1$ and the transitive action of $SL(2,\C)$, the same applies for the flags based at any point on $L^+$. It follows that $\Phi_1$ is smooth, surjective and 2--1. Moreover, the stabiliser of a flag in $SO(1,3)^+$ is trivial, so that $PSL(2,\C)$ acts freely and transitively on $\F$. Topologically $\Phi_1$ is a map $\C_*^2 \cong S^3 \times \R \To \RP^3 \times \R \cong \F$ which is a double cover.

\section{From Minkowski space to horospheres}
\label{Sec:Minkowski_horospheres}

We have now built the maps $\phi_1$ and $\Phi_1$ in the commutative diagram
\[
\begin{array}{ccccc}
\C_*^2 & \stackrel{\Phi_1}{\To} & \F & \stackrel{\Phi_2}{\To} & \Hor^D \\
& \stackrel{\phi_1}{\searrow} & \downarrow & & \downarrow \\
& & L^+ & \stackrel{\phi_2}{\To} & \Hor
\end{array}
\]
where the downwards arrow $\F \To L^+$ forgets all structure of a flag except its point on $L^+$. In this section we define the maps $\phi_2, \Phi_2$, and the spaces $\Hor, \Hor^D$, which involve horospheres and decorations.

Horospheres in the hyperboloid model $\hyp^3$ are given by the intersection of $\hyp^3$ with certain affine 3-planes in $\R^{1,3}$. Any affine 3-plane in $\R^{1,3}$ is given by $x \in \R^{1,3}$ satisfying an equation of the form $\langle x,n \rangle = c$, where $n$ is a (Minkowski-)normal vector to the plane and $c$ is a real constant. We call such an affine 3-plane \emph{lightlike} if its normal $n$ is lightlike. We observe that a lightlike 3-plane can be defined by an equation $\langle x, p \rangle = c$ where $p \in L^+$; if $c>0$ then this plane intersects $\hyp^3$ in a horosphere, and if $c \leq 0$ the plane is disjoint from $\hyp^3$. Normalising such equations by requiring the constant to be $1$, i.e. $\langle x,p \rangle = 1$, then gives a bijection between points $p \in L^+$ and horospheres. We denote the set of horospheres in $\hyp^3$ by $\Hor$.

\begin{defn}[\cite{Penner87}]
The map $\phi_2 \colon L^+ \To \Hor$ sends $p \in L^+$ to the horosphere defined by $\langle x, p \rangle = 1$. The map $\phi_2^\partial \colon L^+ \To \partial \hyp^3$ sends $p$ to the point at infinity of $\phi_2 (p)$.
\end{defn}
Thus the map $\phi_2$ is a bijection. Indeed, it is a diffeomorphism: $\Hor \cong S^2 \times \R$, with an $\R$-family of horospheres at each point at infinity in $S^2 = \partial \hyp^3$. Any horosphere has a unique point at infinity in $\partial \hyp^3$, which we also call its \emph{centre}. The map $\phi_2^\partial$ can be regarded as the projectivisation map $L^+ \To S^2$ or projection to the celestial sphere $L^+ \To \S^+$.

Note $SL(2,\C)$ acts naturally on $\hyp^3$ (as on $L^+$ and $\R^{1,3}$) in the standard way, via linear maps of $SO(1,3)^+$, and hence also on $\partial \hyp^3$ and $\Hor$, and we again denote all actions via a dot. We observe an $A \in SL(2,\C)$ sends the horosphere $\phi_2 (p)$, defined by $\langle x,p \rangle = 1$, to the horosphere $A.\phi_2 (p)$ defined by $\langle A^{-1} x, p \rangle = 1$. Since the action of $A$ preserves the Minkowski metric, this horosphere is also given by $\langle x, A.p \rangle = 1$. In other words, $A.\phi_2 (p) = \phi_2 (A.p)$ so that $\phi_2$ is $SL(2,\C)$-equivariant. Forgetting the horospheres and recording only points at infinity, similarly $\phi_2^\partial$ is $SL(2,\C)$-equivariant.

We now consider the intersection of a horosphere with a flag. So consider a horosphere $\phi_2 (p)$ for some $p \in L^+$, and consider a flag based at the same $p \in L^+$, given by the oriented sequence $\{0\} \subset \R p \subset V \subset p^\perp$. The horosphere $\phi_2 (p)$ is the intersection of $\hyp^3$, given by $\langle x,x \rangle = 1$, with the plane $\langle x, p \rangle = 1$; hence at a point $q \in \phi_2 (p)$, its tangent space is given by $T_q \phi_2 (p) = p^\perp \cap q^\perp$. The intersection of the horosphere with the flag plane $V$ at $q$ will thus be given by
\[
T_q \phi_2 (p) \cap V = q^\perp \cap p^\perp \cap V = q^\perp \cap V,
\]
since $V \subset p^\perp$. Now the intersection $q^\perp \cap V$ is the intersection of a spacelike 3-plane $q^\perp = T_q \hyp^3$, and the 2-plane $V$, so it is either 1- or 2-dimensional. But if it were 2-dimensional then we would have $V \subset q^\perp = T_q \hyp^3$; but $V$ contains a lightlike vector $p$, while $q^\perp = T_q \hyp^3$ is spacelike. Thus the intersection is 1-dimensional and spacelike.

Moreover, the orientation on $\R p \subset V$ is an orientation on $V/\R p$, and thus any vector in $V$ not in $\R p$ obtains an orientation, depending on the side of $\R p$ to which it lies. The intersection $T_q \phi_2 (p) \cap V = q^\perp \cap V$ is spacelike, hence not equal to $\R p$. Thus we may regard the intersection of the horosphere $\phi_2 (p)$ with the flag plane $V$ as defining an oriented line tangent to the horosphere at each point. In other words, we obtain an \emph{oriented line field} on $\phi_2 (p)$. We denote by $\Hor^L$ the set of horospheres with oriented line fields.

\begin{defn}
The map $\Phi_2 \colon \F \To \Hor^L$ sends a flag $\{0\} \subset \R p \subset V$ to the horosphere $\phi_2 (p)$, with the oriented line field defined at each point $q$ by $T_q \phi_2 (p) \cap V$.
\end{defn}
An $A \in SL(2,\C)$ acts on $\Hor^L$: linear maps in $SO(1,3)^+$ are orientation-preserving isometries of $\hyp^3$, sending horospheres to horospheres, with their derivatives sending oriented line fields to oriented line fields. Since the $SL(2,\C)$-actions on $\F$ and $\Hor^L$ are both via linear maps of $SO(1,3)^+$ acting on $\R^{1,3}$, $\Phi_2$ is $SL(2,\C)$-equivariant.

It is well known that a horosphere $H$ is isometric to a Euclidean 2-plane. The parabolic orientation-preserving isometries of $\hyp^3$ fixing $H$ act as translations on this 2-plane. This group of translations is isomorphic to the additive complex numbers. Thus, the following notion of parallelism makes sense.
\begin{defn}
An oriented line field on a horosphere $H$ is \emph{parallel} if it is invariant under Euclidean translations (i.e. under the action of all parabolic isometries fixing $H$).

A \emph{decorated horosphere} is a horosphere with a parallel oriented line field. The set of all decorated horospheres is denoted $\Hor^D$.
\end{defn}
Observe that to describe a parallel oriented line field on a horosphere, it suffices to give an oriented tangent line at one point; the rest of the oriented line field can then be found by parallel translation.

The following lemma calculates $\Phi_2$ for a simple but useful example.
\begin{lem}
\label{Lem:horosphere_line_field_calculation}
$\Phi_2 \circ \Phi_1 (1,0)$ is the horosphere $H_0$ in $\hyp^3$ which has point at infinity in the direction $p_0 = (1,0,0,1)$ along $L^+$, passing through $q_0 = (1,0,0,0)$, with the oriented parallel line field pointing in the direction $\partial_Y = (0,0,1,0)$ at $q_0$.
\end{lem}

\begin{proof}
We have $\phi_1 (1,0) = p_0$, so that $\phi_2 \circ \phi_1 (1,0) =\phi_2 (p_0)$ is the horosphere given by $\langle x, p_0 \rangle = 1$, which is indeed the horosphere $H_0$. From \reflem{Phi1_calculation} the flag $\Phi_1 (1,0)$ is given by $[[ p_0, \partial_Y ]]$, so the flag 2-plane $V$ is spanned by $p_0$ and $\partial_Y$, with relative orientation on $V/\R p_0$ given by $\partial_Y$.

Now the parabolic subgroup 
\begin{equation}
\label{Eqn:parabolic_subgroup}
P = \left\{ \begin{bmatrix} 1 & c \\ 0 & 1 \end{bmatrix} \; \mid \; c \in \C \right\}
\end{equation}
fixes $(1,0)$ and acts simply transitively on $H_0$. Denoting by $A_c$ the matrix in $P$ with upper right entry $c \in \C$, the points of $H_0$ are parametrised by $c \in \C$; letting $q_c = A_c.q_0$ we have $H_0 = \{ q_c \; \mid \; c \in \C \}$. We calculate the action of $A_c$ on $\R^{1,3}$ to be
\[
A_c.(T,X,Y,Z) = (T',X',Y',Z')
\]
where
\[
\begin{array}{cc}
T' = T+ \Re c \, X + \Im c \, Y + \frac{|c|^2}{2}(T-Z), 
& X' = X + \Re c \, (T-Z),  \\
Y' = Y + \Im c \, (T-Z),
&Z' = Z + \Re c \, X + \Im c \, Y + \frac{|c|^2}{2}(T-Z).
\end{array}
\]
Thus we calculate
\[
q_c = A_c.q_0 =  \left( 1 + \frac{|c|^2}{2}, \; \Re c, \; \Im c, \; \frac{|c|^2}{2} \right)
\]
and moreover for $c,c' \in \C$ we have $A_c . q_{c'} = q_{c+c'}$.
At $q_c$ the line field of $\Phi_2 (p_0)$ is given by $q_c^\perp \cap V$. Now $q_c^\perp$ is the 3-plane given by equation
\begin{equation}
\label{Eqn:qcperp}
\left( 1 + \frac{|c|^2}{2} \right) T - \Re c \; X - \Im c \; Y - \frac{|c|^2}{2} Z = 0,
\end{equation}
while $V$ is spanned by $p_0$ and $\partial_Y$, hence defined by $T=Z$ and $X=0$. Thus $q_c^\perp \cap V$ is defined by $T=Z$, $X=0$ and $Z = (\Im c) \, Y$, hence spanned by $(\Im c, 0, 1, \Im c) = (\Im c) \, p_0 + \partial_Y$. Since the orientation on $V/\R p_0$ is given by $\partial_Y$, the oriented line field of $\Phi_2 \circ \Phi_1 (1,0)$ at $q_c$ is directed by $(\Im c) \, p_0 + \partial_Y$. In particular, the oriented line field of $\Phi_2 \circ \Phi_1 (1,0)$ at $q_0$ is directed by $\partial_Y$. 

Now, if we apply $A_c$ to the vector $(\Im c', 0,1, \Im c')$ directing the line field at a point $q_{c'}$ of $H_0$, we obtain the vector $(\Im (c+c'), 0, 1, \Im (c+c'))$ at $q_{c+c'}$. Thus the oriented line field is parallel.
\end{proof}
In fact in the above calculation we observe that $A_c.p_0 = p_0$ and $A_c.\partial_Y = \partial_Y + \Im c \, p_0$. This shows explicitly that the parabolic subgroup $P$ preserves the flag plane $V$, and in fact acts as the identity on both $\R p_0$ and $V/\R p_0$.

In fact this example is generic enough to give the following.
\begin{lem}
\label{Lem:Phi_2_diffeo}
The map $\Phi_2$ is a diffeomorphism $\F \To \Hor^D$. 
\end{lem}

In other words, the oriented line field of any flag is parallel, and $\Phi_2$ provides a smooth correspondence between flags and decorated horospheres.

\begin{proof}
First we show $\Phi_2$ always yields parallel oriented line fields. \reflem{horosphere_line_field_calculation} shows this when $\Phi_2$ is applied to $\Phi_1 (1,0) \in \F$. But the action of $SL(2,\C)$ is transitive on $\F$, and the action of $SL(2,\C)$ on $\hyp^3$ (hence on horospheres) is by isometries, and $\Phi_2$ is $SL(2,\C)$-equivariant. Any flag in $\F$ is thus of the form $A.\Phi_1(1,0)$ for some $A \in SL(2,\C)$, so $\Phi_2 A. \Phi_1 (1,0) = A.\Phi_2 \Phi_1 (1,0)$, which has parallel oriented line field.

Next we show that $\Phi_2$ sends the flags of the form $[[p_0, v]]$ bijectively to decorations on $H_0$. These flags are those of the form 
\[
\Phi_1 (e^{i\theta},0) = \begin{bmatrix} e^{i\theta} & 0 \\ 0 & e^{-i\theta} \end{bmatrix}.\Phi_1 (1,0)
= [[p_0, (0, -\sin 2\theta, \cos 2\theta, 0)]],
\]
as calculated above. Denote the latter vector by $\partial_\theta$, so $\Phi_1 (e^{i\theta},0) = [[p_0, \partial_\theta]]$. Then $\Phi_2 \Phi_1 (e^{i\theta},0)$ is the horosphere $H_0$, with oriented parallel line field at $q_0$ given by the intersection of the flag 2-plane with $q_0^\perp$. Since $q_0^\perp$ is given by $T=0$ (equation (\refeqn{qcperp}) with $c=0$), which contains $\partial_\theta$, the oriented line field of $\Phi_2 \Phi_1 (e^{i\theta},0)$ at $q_0$ is directed by $\partial_\theta$. As $\theta$ increases say from $0$ to $\pi$, both the flag through $p_0$ and the decoration on $H_0$ rotate through a full $2\pi$, with $\Phi_2$ providing a bijection. 

We have already seen that $\phi_2$ provides a bijection between $L^+$ and $\Hor$; using the transitivity of $SL(2,\C)$ on $\F$ and $\Hor^D$, and equivariance of $\Phi_2$, it follows that $\Phi_2$ provides a bijection between the flags based at each $p_0 \in L^+$, and decorations on the corresponding horospheres $\phi_2 (p_0)$.

Thus $\Phi_2$ is a bijection. It and its inverse are clearly smooth, once $\F$ and $\Hor^D$ are given their natural smooth structures. We have already seen $\F\cong UTS^2 \times \R$. The space of horospheres is naturally $\Hor \cong S^2 \times \R$, and decorations can be given by unit tangent vectors to the sphere at infinity, so that $\Hor^D \cong UTS^2 \times \R$.
\end{proof}

We now consider our horospheres in the upper half space model $\U^3$ of hyperbolic 3-space, given in the usual way as
\[
\U^3 = \left\{ (x,y,z) \in \R^3 \; \mid \; z>0 \right\}
\quad \text{with metric} \quad
ds^2 = \frac{dx^2+dy^2+dz^2}{z^2}.
\]
As usual we identify the plane $z=0$ with $\C$ and $\partial \U^3$ with $\C \cup  \{  \infty  \} $, and coordinates $(x,y)$ with $x+yi \in \C$. In $\U^3$, horospheres centred at $\infty$ appear as horizontal planes; we call the $z$-coordinate of this plane the \emph{height} of the horosphere. Horospheres centred at other points appear as spheres tangent to $\C$; we call the maximum of $z$ on the sphere the \emph{Euclidean diameter} of the horosphere.

We proceed from $\hyp^3$ to $\U^3$ via the conformal ball model $\Disc^3$. We have the standard isometries given by
\begin{equation}
\label{Eqn:hyp_models_translations}
\hyp^3 \To \Disc^3, \quad (T,X,Y,Z) \mapsto \frac{1}{1+T} \left( X,Y,Z \right)
\end{equation}
and, on the boundaries,
\begin{equation}
\label{Eqn:hyp_models_translations_2}
\partial \Disc^3 \To \partial \U^3, \quad (x,y,z) \mapsto \frac{x+iy}{1-z}.
\end{equation}
where in the latter map we regard $\partial \Disc^3$ as the standard $S^2 \subset \R^3$ and $\partial \U^3$ as $\C \cup \{\infty\}$. Of course $SL(2,\C)$-actions carry through equivariantly to each model as isometries, and on $\partial \U^3 \cong \C \cup  \{  \infty  \} $ the action is via M\"{o}bius transformations in the usual way,
\[
\begin{bmatrix} \alpha & \beta \\ \gamma & \delta \end{bmatrix}.\zeta 
= \frac{\alpha \zeta + \beta}{\gamma \zeta + \delta}
\quad \text{for $\zeta \in \C$}.
\]

We now introduce some terminology to describe decorations, i.e. parallel oriented line fields, on horospheres in $\U^3$. A horosphere centred at $\infty$ is a horizontal plane parallel to $\C$, so a parallel oriented line field appears as a line field invariant under Euclidean translations, and can be described by a complex number which points in the direction of the lines. This complex number is well defined up to positive multiples and we say it \emph{specifies} the decoration. On a horosphere centred elsewhere, we can describe an oriented line field by giving a vector directing it at the point with maximum $z$-coordinate (its ``north pole"); since the tangent plane there is also parallel to $\C$, we can also describe it by a complex number, up to positive multiple. We call this a \emph{north pole specification} of a decoration.

We can now give the decorated horospheres corresponding to spinors explicitly, verifying the description in the introduction, illustrated in \reffig{2} (right).
\begin{prop}
\label{Prop:horospheres_explicitly}
The spinor $(\xi, \eta)$ maps under $\Phi_2 \circ \Phi_1$ to a decorated horosphere whose centre is at $\xi/\eta$ in the upper half space model. 
\begin{enumerate}
\item
If $\eta \neq 0$ then the horosphere has Euclidean diameter $|\eta|^{-2}$, and decoration north-pole specified by $i/\eta^2$,
\item
If $\eta = 0$, then the horosphere has height $|\xi|^2$, and decoration specified by $i \xi^2$.
\end{enumerate}
\end{prop}

In particular, forgetting the decorations, the above proposition gives an explicit description of $\phi_2 \circ \phi_1 (\xi, \eta)$. And forgetting all but the centres of the horospheres, it yields $\phi_2^\partial \circ \phi_1 (\xi, \eta) = \frac{\xi}{\eta}$.

\begin{proof}
Letting $\xi = a+bi$, $\eta = c+di$, $\phi_1(\xi,\eta)$ is given in \reflem{Phi1_calculation}. Then $\phi_2^\partial$, for the hyperboloid model, just projectivises the rays of $L^+$ to points; taking $(X,Y,Z)$ for the point on each ray with $T=1$ gives $\phi_2^\partial \phi_1 (\xi, \eta)$ on $\partial \Disc$ as
\[
\frac{1}{a^2+b^2+c^2+d^2} \left( 2(ac+bd), \, 2(bc-ad), \, a^2 + b^2 - c^2 - d^2 \right).
\]
The centre of the horosphere on $\partial \U \cong \C \cup  \{ \infty  \}$ is then, using \refeqn{hyp_models_translations_2},
\[
\frac{ \frac{2(ac+bd)}{a^2 + b^2 + c^2 + d^2} + i \frac{2(bc-ad)}{a^2 + b^2 + c^2 + d^2}}{1 - \frac{a^2 + b^2 - c^2 - d^2}{a^2+b^2+c^2+d^2}} = \frac{a+bi}{c+di} = \frac{\xi}{\eta}.
\]

From \reflem{horosphere_line_field_calculation}, $\Phi_2 \circ \Phi_1 (1,0)$, in the hyperboloid model, is the horosphere centred at $p_0 = (1,0,0,1)$, passing through $q_0 = (1,0,0,0)$, and at $q_0$ has decoration in the direction $\partial_Y = (0,0,1,0)$. In $\Disc$, this corresponds to the horosphere centred at $(0,0,1)$, passing through $(0,0,0)$, and having decoration in the direction $(0,1,0)$ there. In $\U$, this corresponds to the horosphere centred at $\infty$, passing through $(0,0,1)$, and having decoration in the direction $(0,1,0)$ at that point. In other words, it has height $1$ and decoration specified by $i$.

The decorated horospheres $\Phi_2 \circ \Phi_1 (\xi, \eta)$ can now be found in general using $SL(2,\C)$-equivariance. Observe that
\begin{equation}
\label{Eqn:matrices_to_everywhere}
\begin{pmatrix} 0 \\ 1 \end{pmatrix} = \begin{bmatrix} 0 & -1 \\ 1 & 0 \end{bmatrix}.\begin{pmatrix} 1 \\ 0 \end{pmatrix}, \quad
\begin{pmatrix} \xi \\ 0 \end{pmatrix} = \begin{bmatrix} \xi & 0 \\ 0  & \xi^{-1} \end{bmatrix}.\begin{pmatrix} 1 \\ 0 \end{pmatrix}, \quad
\begin{pmatrix} \xi \\ \eta \end{pmatrix} = \begin{bmatrix} \eta^{-1} & \xi \\ 0 & \eta \end{bmatrix}.\begin{pmatrix} 0 \\ 1 \end{pmatrix}
\end{equation}
Thus the decorated horosphere of $(0,1)$ is obtained from the decorated horosphere of $(1,0)$ by applying the M\"{o}bius transformation $z \mapsto \frac{-1}{z}$; hence it is centred at $0$, has Euclidean diameter $1$, and is north-pole specified by $i$. Similarly, the decorated horosphere of $(\xi,0)$, for $\xi \neq 0$, is obtained from that of $(1,0)$ by applying $z \mapsto \xi^2 z$, hence is centred at $\infty$, has height $|\xi|^2$, and is specified by $i \xi^2$. And the decorated horosphere of $(\xi,\eta)$, for $\eta \neq 0$, is obtained from that of $(0,1)$ by applying $z \mapsto \eta^{-2} z + \frac{\xi}{\eta}$, hence is centred at $\xi/\eta$, has Euclidean diameter $|\eta|^{-2}$, and is north-pole specified by $i \eta^{-2}$.
\end{proof}

Thus, if we multiply a spinor $\kappa$ by a complex number $re^{i\theta}$, with $r>0$ and $\theta$ real, the effect on the corresponding horosphere $H$ is to translate it by distance $2 \log r$ along any geodesic $\gamma$ perpendicular to $H$ oriented towards its centre, and rotate the decoration by $2 \theta$ about $\gamma$.

\section{Spin decorations and complex lambda lengths}
\label{Sec:spin_decorations}

We now introduce the concepts necessary to explain the lifts of previous constructions to spin double covers, and the notion of complex lambda length between two spin-decorated horospheres. In this section $\hyp^3$ refers to hyperbolic 3-space, regardless of model.

We use the cross product $\times$ in $\hyp^3$ in the elementary sense that if $v,w$ are tangent to $\hyp^3$ at a common point $p$, making an angle of $\theta$, then $v \times w$ is tangent to $\hyp^3$ at $p$, has length $|v| \, |w| \sin \theta$, and points in the direction perpendicular to $v$ and $w$ given by the right-hand rule.

A horosphere $H$ in $\hyp^3$ (like any oriented surface in a 3-manifold) has two normal directions: we call the direction towards its centre \emph{outward} (``pointing out of $\hyp^3$"), and the direction away from its centre \emph{inward} (``pointing into $\hyp^3$"). There are well-defined outward and inward unit normal vector fields along $H$, which we denote $N^{out}, N^{in}$ respectively.

By a \emph{frame} we mean a right-handed orthonormal frame at a point in $\hyp^3$, i.e. a triple of orthogonal unit vectors $(f_1, f_2, f_3)$ such that $f_1 \times f_2 = f_3$. The collection of frames then forms a principal $SO(3)$-bundle over $\hyp^3$ which we denote
\[
\Fr \To \hyp^3.
\]
We may take its spin double (universal) cover, which we denote
\[
\Fr^S \To \hyp^3,
\]
which is a principal $\Spin(3)$-bundle. We refer to points of $\Fr^S$ as \emph{spin frames}. Each point in $\Fr$ has two lifts in $\Fr^S$, i.e. each frame lifts to two spin frames.

The group of orientation-preserving symmetries of $\hyp^3$ is naturally isomorphic to $PSL(2,\C)$, and acts simply transitively on $\Fr$. Choosing a basepoint $F_0$ in $\Fr$ then we may obtain an explicit identification $PSL(2,\C) \cong \Fr$, given by $M \leftrightarrow M.F_0$ for $M \in PSL(2,\C)$. 

Similarly, $SL(2,\C)$ acts simply transitively on $\Fr^S$. And the identification $PSL(2,\C) \cong \Fr$ lifts to double covers, after we choose a lifted basepoint $\widetilde{F_0}$, giving an explicit diffeomorphism $SL(2,\C) \cong \Fr^S$ as $A \leftrightarrow A.\widetilde{F_0}$. The two matrices $A,-A \in SL(2,\C)$ lifting $\pm A \in PSL(2,\C)$ then correspond to the two spin frames $A.F_0$, $-A.F_0
$ lifting the frame $(\pm A).F_0$. These two spin frames are related by a $2\pi$ rotation about any axis at their common point.
We can regard elements of $SL(2,\C)$ as \emph{spin isometries}; each isometry in $PSL(2,\C)$ lifts to two spin isometries, which differ by a $2\pi$ rotation.
Since $SL(2,\C)$ is the universal cover of the isometry group $PSL(2,\C)$, we can also regard elements of $SL(2,\C)$ as homotopy classes of paths of isometries starting at the identity.

From a decoration on a horosphere $H$, normalised to a parallel unit tangent vector field $v$ on $H$, we can then construct frame fields along $H$ as follows.
\begin{defn}
Let $v$ be a unit parallel tangent vector field on a horosphere $H$. 
\begin{enumerate}
\item
The \emph{inward frame field} of $v$ is the frame field on $H$ given by $F^{in} = (N^{in}, v, N^{in} \times v)$.
\item
The \emph{outward frame field} of $v$ is the frame field on $H$ given by $F^{out} = (N^{out}, v, N^{out} \times v)$.
\end{enumerate}
Indeed a decorated horosphere is uniquely specified by its inward and outward frame fields and so we can denote a decorated horosphere by $(H, F)$ where $F$ is the pair of frames $F = (F^{in}, F^{out})$.
\end{defn}
A frame field is a continuous section of $\Fr$ along $H$, and it has two lifts to $\Spin$.
\begin{defn}
An \emph{outward (resp. inward) spin decoration} on $H$ is a continuous lift of an outward (resp. inward) frame field from $\Fr$ to $\Fr^S$.
\end{defn}
From the inward frame field $(N^{in}, v, N^{in} \times v)$ of a unit parallel vector field $v$ on $H$, one can rotate the frame at each point of $H$ by an angle of $\pi$ or $-\pi$ about $v$ to obtain the outward frame field of $v$, and vice versa. After taking an inward spin decoration lifting the inward frame field, one can similarly rotate the frame at each point by an angle of $\pi$ about $v$, which will result in an outward spin decoration. However, rotations of $\pi$ or $-\pi$ about $v$ yield distinct results, related by a $2\pi$ rotation. Thus we make the following definition, which is a somewhat arbitrary convention, but we need it for our results to hold.
\begin{defn} \
\label{Def:associated_spin_decorations}
\begin{enumerate}
\item
Let $W^{out}$ be an outward spin decoration on $H$. The \emph{associated inward spin decoration} is the spin decoration obtained by rotating $W^{out}$ by angle $\pi$ about $v$ at each point of $H$.
\item
Let $W^{in}$ be an inward spin decoration on $H$. The \emph{associated outward spin decoration} is the spin decoration obtained by rotating $W^{in}$ by angle $-\pi$ about $v$ at each point of $H$.
\end{enumerate}
\end{defn}
We observe that associated spin decorations come in pairs $W = (W^{in}, W^{out})$, each associated to the other.
\begin{defn}
A \emph{spin decoration} on a horosphere $H$ is a pair $W = (W^{in}, W^{out})$ of associated inward and outward spin decorations. We denote a spin-decorated horosphere by $(H, W)$, and denote the set of spin-decorated horospheres by $\Hor^S$.
\end{defn}

Note that under the identification $PSL(2,\C) \cong \Fr$, with an appropriate choice of basepoint frame, the parabolic subgroup $P$ of equation \refeqn{parabolic_subgroup} (or more precisely its image $\pm P$ in $PSL(2,\C)$) corresponds to all the frames of the outward frame field of $\Phi_2 \circ \Phi_1 (1,0)$. The cosets of $\pm P$ then correspond bijectively with decorated horospheres. Similarly, under the identification $SL(2,\C) \cong \Fr^S$ with an appropriate choice of basepoint, the cosets of $P$ correspond bijectively with spin-decorated horospheres:
\[
PSL(2,\C) / (\pm P) \cong \Hor^D, \quad
SL(2,\C) / P \cong \Hor^S.
\]

We now consider lifts of the maps 
\[
\C_*^2 \stackrel{\Phi_1}{\To} \F \stackrel{\Phi_2}{\To} \Hor^D
\]
Topologically, we have $\C_*^2 \cong S^3 \times \R$, we have seen $\F \cong \Hor^D \cong UTS^2 \times \R \cong \RP^3 \times \R$, and we have seen $\Phi_1$ is a double cover and $\Phi_2$ is a diffeomorphism. Indeed, all the spaces here are $S^1 \times \R \cong \C_*$ bundles over $S^2$ and the maps are bundle maps, which in an appropriate sense are the identity on the base space $S^2$. The spaces $\F$ and $\Hor^D$ both have fundamental group $\Z/2$, and we can consider their double (hence universal) covers. A nontrivial loop in $\F$ is given by fixing a flagpole and rotating a flag through $2\pi$; in the double cover, rotating the flag through $2\pi$ is no longer a loop, but rotating the flag through $4\pi$ gives a loop.  

\begin{defn}
The double cover of the space of flags $\F$ is denoted $\F^S$. We call its elements \emph{spin flags}.
\end{defn}
Our spin flags are the \emph{null flags} of \cite{Penrose_Rindler84}.

A nontrivial loop in $\Hor^D$ is given by fixing a horosphere and rotating its decoration through $2\pi$. In the double cover, a rotation through $2\pi$ is not a loop  but a rotation through $4\pi$ gives a loop. In other words, the double cover of $\Hor^D$ is $\Hor^S$. Choosing basepoints (arbitrarily) one then obtains lifts $\widetilde{\Phi_1}, \widetilde{\Phi_2}$ such that the diagram
\[
\begin{array}{ccccc}
\C_*^2 & \stackrel{\widetilde{\Phi_1}}{\To} & \F^S & \stackrel{\widetilde{\Phi_2}}{\To} & \Hor^S \\
& \stackrel{\Phi_1}{\searrow} & \downarrow & & \downarrow \\
& & \F & \stackrel{\Phi_2}{\To} & \Hor^D
\end{array}
\]
commutes, where the downwards arrows are double covering maps. The action of $SL(2,\C)$ (which is simply connected) lifts to actions on these covers and all maps remain $SL(2,\C)$-equivariant.

Thus a spinor $\kappa$ maps under $\widetilde{\Phi_2} \circ \widetilde{\Phi_1}$ to a spin-decorated horosphere lifting the decorated horosphere described in \refprop{horospheres_explicitly}. Multiplying $\kappa$ by $re^{i\theta}$, with $r>0$ and $\theta$ real, still translates it $2 \log r$ towards its centre and rotates the decoration by $2 \theta$, but now the rotation is taken modulo $4\pi$.

We can now prove \refthm{main_thm_1}, that there is an explicit smooth bijective correspondence between $\C_*^2$ and $\Hor^S$.
\begin{proof}[Proof of \refthm{main_thm_1}]
At the end of \refsec{spin_vectors_to_Minkowski} we observed that $\Phi_1$ is a smooth double cover, topologically $S^3 \times \R \To \RP^3 \To \R$. In \reflem{Phi_2_diffeo} we showed $\Phi_2$ is a diffeomorphism. Their lifts $\widetilde{\Phi_1}$ and $\widetilde{\Phi_2}$ are then both diffeomorphisms, topologically $S^3 \times \R \To S^3 \times \R$. We have defined these maps explicitly. We have also shown all maps are $SL(2,\C)$-equivariant. Thus $\widetilde{\Phi_2} \circ \widetilde{\Phi_1}$ provides the claimed correspondence.
\end{proof}

We use spin frames to define complex lambda lengths between spin-decorated horospheres. For this, we need to compare frames along geodesics, and we need frames to be adapted to geodesics, in a suitable sense. (Here, as throughout, frames are right-handed and orthonormal.)
\begin{defn}
Let $p$ be a point on an oriented geodesic $\gamma$ in $\hyp^3$. A frame $F = (f_1, f_2, f_3)$ at $p$ is \emph{adapted} to $\gamma$ if $f_1$ is positively tangent to $\gamma$. A spin frame $\widetilde{F}$ at $p$ is \emph{adapted} to $\gamma$ if it is the lift of a frame adapted to $\gamma$.
\end{defn}

Now if we have two points $p_1, p_2$ on an oriented geodesic $\gamma$, and frames $F^i = (f_1^i, f_2^i, f_3^i)$ at each $p_i$, adapted to $\gamma$, then there is then a screw motion along $\gamma$ which takes $F^1$ to $F^2$ as follows. Being adapted to $\gamma$, the first vectors $f_1^1$ and $f_1^2$ in each frame point along $\gamma$. Parallel translation along $\gamma$ from $p_1$ to $p_2$ takes $F^1$ to a frame $F'^1$ at $p_2$ which agrees with $F^2$ in its first vector. This translation is by a signed distance $\rho$ which we regard as positive or negative according to the orientation on $\gamma$. A further rotation of some angle $\theta$ about $\gamma$ (signed using the orientation of $\gamma$) then moves $F'^1$ to $F^2$. Note that $\theta$ is only well defined modulo $2\pi$. However we may repeat this process with spin frames, and then $\theta$ is well defined modulo $4\pi$.
\begin{defn}
Let $F^1, F^2$ be frames, or spin frames, at points $p_1, p_2$ on an oriented geodesic $\gamma$, adapted to $\gamma$. The \emph{complex distance} from $F^1$ to $F^2$ is $\rho + i \theta$, where a translation along $\gamma$ of signed distance $\rho$, followed by a rotation about $\gamma$ of angle $\theta$, takes $F^1$ to $F^2$.
\end{defn}
In general two frames are not adapted to a common oriented geodesic, but when two frames are adapted to a common oriented geodesic, that oriented geodesic is unique, and so we may speak of the complex distance between the frames. The same applies to spin frames. Note that the complex distance between frames adapted to a common geodesic is well defined modulo $2\pi i$; between spin frames, it is well defined modulo $4\pi i$.

We can now define complex lambda lengths between decorated and spin-decorated horospheres. Let $H_1, H_2$ be horospheres, let $z_i \in \partial \hyp^3$ be the centre of $H_i$, and let $\gamma_{ij}$ be the oriented geodesic from $z_i$ to $z_j$. Thus $\gamma_{12}$ and $\gamma_{21}$ are the two orientations of the unique common perpendicular to the horospheres. Let $p_i = \gamma_{12} \cap H_i$. If the $H_i$ are decorated, we have pairs $F_i = (F^{in}_i, F^{out}_i)$ of inward and outward frame fields on each $H_i$, and note that $F^{in}_1 (p_1)$ and $F^{out}_2 (p_2)$ are both adapted to $\gamma_{12}$. If the $H_i$ are spin-decorated, we have pairs $W_i = (W^{in}_i, W^{out}_i)$ of associated inward and outward spin decorations on each $H_i$, and we note that $W^{in}_1 (p_1)$ and $W^{out}_2 (p_2)$ are adapted to $\gamma_{12}$.
\begin{defn} \
\begin{enumerate}
\item
If $(H_1, F_1)$ and $(H_2, F_2)$ are decorated horospheres, the \emph{complex lambda length} from $(H_1, F_1)$ to $(H_2, F_2)$ is
\[
\lambda_{12} = \exp \left( \frac{d}{2} \right),
\]
where $d$ is the complex distance from $F^{in}_1 (p_1)$ to $F^{out}_2 (p_2)$.
\item
If $(H_1, W_1)$ and $(H_2, W_2)$ are spin-decorated horospheres, the \emph{complex lambda length} from $(H_1, W_1)$ to $(H_2, W_2)$ is
\[
\lambda_{12} = \exp \left( \frac{d}{2} \right),
\]
where $d$ is the complex distance from $W^{in}(p_1)$ to $W^{out}(p_2)$.
\end{enumerate}
When the horospheres $H_1$ and $H_2$ have a common centre, then the complex lambda length between them is zero in either case.
\end{defn}

Note that for decorated horospheres, $d$ is only well defined modulo $2\pi i$, so $\lambda_{12}$ is only well defined up to sign. For spin-decorated horospheres however $d$ is well defined modulo $4\pi i$, so $\lambda_{12}$ is a well defined complex number, and indeed we have a well defined function $\lambda \colon \Hor^S \times \Hor^S \To \C$.

We observe that $\lambda$ is in fact continuous. In particular, if two horospheres move so that their centres approach each other, then the length of the segment of their common perpendicular geodesic which lies in the intersection of the horoballs becomes arbitrarily large, so $\Re d \rightarrow -\infty$ and hence $\lambda \rightarrow 0$.

In fact, as we now see, $\lambda$ is antisymmetric. This lemma is also a corollary of \refthm{main_thm_2}, so strictly speaking is unnecessary, but the proof below gives a direct geometric explanation.
\begin{lem}
Let $(H_1, W_1)$, $(H_2, W_2)$ be spin-decorated horospheres, and let $\lambda_{ij}$ be the complex lambda length from $(H_i, W_i)$ to $(H_j, W_j)$. Then $\lambda_{12} = -\lambda_{21}$.
\end{lem}

\begin{proof}
If $H_1, H_2$ have common centre $\lambda_{12} = \lambda_{21} = 0$. So we may assume $H_1, H_2$ have distinct centres $z_1, z_2$. As above, let $\gamma_{ij}$ be the oriented geodesic from $z_1$ to $z_2$, and let $p_i = \gamma_{12} \cap H_i$. Let $d_{ij}$ be the complex distance from $W_i^{in}$ to $W_j^{out}$ along $\gamma_{ij}$. The spin frames $W_i^{in}$, $W_i^{out}$ yield frames $F_i^{in}, F_i^{out}$ of unit parallel vector fields $V_i$ on $H_i$. 

Recall from \refdef{associated_spin_decorations} that $W_2^{in}$ is obtained from $W_2^{out}$ by a rotation of $\pi$ about $V_2$, and $W_1^{out}$ is obtained from $W_1^{in}$ by a rotation of $-\pi$ about $V_1$. Define $Y_1^{out}$ to be the result of rotating $W_1^{in}$ by $\pi$ about $V_1$, so $Y_1^{out}$ and $W_1^{out}$ both project to $F_1^{out}$, but differ by a $2\pi$ rotation.

The spin isometry which takes $W_1^{in} (p_1)$ to $W_2^{out} (p_2)$ also takes $Y_1^{out} (p_1)$  to $W_2^{in} (p_2)$. Hence the complex distance $d_{12}$ from $W_1^{in}(p_1)$ to $W_2^{out}(p_2)$ along $\gamma_{12}$ is equal to the complex distance from $W_2^{in} (p_2)$ to $Y_1^{out} (p_1)$ along $\gamma_{21}$. But since $Y_1^{out}$ and $W_1^{out}$ differ by a $2\pi$ rotation, this latter complex distance is $d_{21} + 2\pi i$. From $d_{12} = d_{21} + 2\pi i$ mod $4\pi i$ we obtain $\lambda_{12} = - \lambda_{21}$.
\end{proof}

If we have two spin frames adapted to a common geodesic, and apply a homotopy $M_t \in PSL(2,\C)$ of isometries to them, for $t \in [0,1]$, starting from the identity $M_0$, the complex distance between the spin frames remains constant; such a homotopy describes a point of the universal cover $SL(2,\C)$. Hence complex distance between spin frames is invariant under the action of $SL(2,\C)$. Similarly applying a homotopy to two spin-decorated horospheres and their common perpendicular geodesic, we observe that complex lambda length is also invariant under the action of $SL(2,\C)$. In other words, if $A \in SL(2,\C)$ and $(H_1, W_1), (H_2, W_2)$ are spin-decorated horospheres, then the complex lambda length from $(H_1, W_1)$ to $(H_2, W_2)$ is equal to the complex lambda length from $A.(H_1, W_1)$ to $A.(H_2, W_2)$.

We can now prove \refthm{main_thm_2}: given spinors $\kappa_1, \kappa_2 \in \C_*^2$, and corresponding spin-decorated horospheres $(H_i, W_i) = \widetilde{\Phi_2} \circ \widetilde{\Phi_1} (\kappa_i)$, the complex lambda length $\lambda_{12}$ form $(H_1, W_1)$ to $(H_2, W_2)$ satisfies
\[
\{ \kappa_1, \kappa_2 \} = \lambda_{12}.
\]

\begin{proof}[Proof of \refthm{main_thm_2}]
Recalling that the spinor $\kappa = (\xi, \eta)$ corresponds to a horosphere in $\U^3$ with centre $\xi/\eta$, we observe that $\kappa_1, \kappa_2$ are linearly dependent (over $\C)$ precisely when $H_1, H_2$ have common centre. In other words, $\{\kappa_1, \kappa_2 \} = 0$ precisely when $\lambda_{12} = 0$. We can thus assume $\kappa_1, \kappa_2$ are linearly independent.

First we prove the result when $\kappa_1 = (1,0)$ and $\kappa_2 = (0,1)$. From \refprop{horospheres_explicitly} then $(H_1, W_1)$ is a spin lift of the decorated horosphere centred at $\infty$ with height $1$ and decoration specified by $i$; and $(H_2, W_2)$ is a spin lift of the decorated horosphere centred at $0$ with Euclidean diameter $1$ and decoration north-pole specified by $i$. They are thus tangent at the point $p = (0,0,1)$, at which point $W_1^{in}$ and $W_2^{out}$ project to coincident frames, hence either coincide or differ by $2\pi$.

To see that they coincide, we consider the following matrix $A \in SL(2,\C)$, regarded as the lift to the universal cover of the path $M_t \in PSL(2,\C)$ for $t \in [0, \pi/2]$, starting at the identity:
\[
A = \begin{bmatrix} 0 & -1 \\ 1 & 0 \end{bmatrix} \in SL(2,\C),
\quad 
M_t = \pm \begin{bmatrix} \cos t & - \sin t \\ \sin t & \cos t \end{bmatrix} \in PSL(2,\C).
\]
Clearly $A.\kappa_1 = \kappa_2$, so by $SL(2,\C)$-equivariance $A.(H_1,W_1)=(H_2,W_2)$. Geometrically, in the upper half space model, $M_t$ is a rotation of angle $2t$ about the oriented geodesic $\delta$ from $-i$ to $i$. Over $t \in [0,\pi/2]$, the point $p$ and the vector $i$ specifying the decorations remain fixed, and the frame $W_1^{in}$ at $p$ rotates by $\pi$ about $\delta$ to arrive at $A.W_1^{in} = W_2^{in}$. Applying \refdef{associated_spin_decorations}, we then obtain the associated outward spin frame $W_2^{out}$ by a rotation of $-\pi$ about the decoration vector, i.e. about the same axis $\delta$. Thus indeed $W_1^{in}(p) = W_2^{out}(p)$, their complex distance is $0$, and $\lambda_{12} = 1$.

Next we prove the result when $\kappa_1 = (1,0)$ and $\kappa_2 = (0,D)$ for some complex $D \neq 0$. In this case $(H_2, W_2)$ is the spin lift of a decorated horosphere centred at $0$, with Euclidean diameter $|D|^{-2}$ and decoration north-pole specified by $i D^{-2}$. The common perpendicular $\gamma_{12}$ runs from $\infty$ to $0$, intersecting $H_1$ at $p_1 = (0,0,1)$ and $H_2$ at $p_2 = (0,0,|D|^{-2})$. Thus the signed translation distance from $p_1$ to $p_2$ is $2 \log |D|$ and the rotation angle is given by $\arg D^2 = 2 \arg D$ mod $2\pi$; lifting to spin frames we show it is indeed $2 \arg D$ mod $4\pi$. Consider again an $A \in SL(2,\C)$ lifting a path $M_t \in PSL(2,\C)$ from the identity, 
\[
A = \begin{bmatrix} e^{-\log|D|-i \arg D} & 0 \\ 0 & e^{\log|D| + i \arg D} \end{bmatrix} = \begin{bmatrix} D^{-1} & 0 \\ 0 & D \end{bmatrix}, \quad
M_t = \pm \begin{bmatrix} e^{-t(\log|D|+i\arg D)} & 0 \\ 0 & e^{t(\log|D|+i\arg D)} \end{bmatrix}
\]
where we take $\arg D \in [0,2\pi)$ and $t \in [0, 1]$. We have $A.(0,1) = \kappa_2$, so $A$ sends $\widetilde{\Phi_2} \circ \widetilde{\Phi_1} (0,1)$ (i.e. $(H_2, W_2)$ from the previous case) to $(H_2, W_2)$ here. Geometrically $M_t$ is a translation of length $2t\log|D|$ and rotation of angle $2t \arg D$ about $\gamma_{12}$, so $A$ as a spin isometry translates by $2 \log |D|$ and rotates by $2 \arg D$ modulo $4\pi$. Since the complex distance from $W_1^{in}$ to $W_2^{out}$ at $p_1$ was zero in the previous case $D=1$, the complex distance now becomes $2 \log |D| + 2 \arg D i$ mod $4\pi i$. Thus $\lambda_{12} = D = \{\kappa_1, \kappa_2\}$.

Finally, we prove the result for general linearly independent $\kappa_1, \kappa_2$. There exists $A \in SL(2,\C)$ such that $A.\kappa_1 = (1,0)$ and $A.\kappa_2 = (0,D)$, where $D = \{\kappa_1, \kappa_2\}$. Applying this $A$ then the complex lambda length from $(H_1, W_1)$ to $(H_2, W_2)$ is equal to the complex lambda length from $A.(H_1, W_1)$ to $A.(H_2, W_2)$, which is $\{\kappa_1, \kappa_2\}$ from the previous case.
\end{proof}

\section{Hyperbolic geometry applications}
\label{Sec:hyperbolic_applications}

The above theory can be applied to any situation involving horospheres in hyperbolic geometry, in up to 3 dimensions. Endowing each horosphere with a spin decoration, we obtain a spinor, and then applying the bilinear form $\{ \cdot, \cdot \}$ gives us geometric information about horospheres.

As a first application we consider hyperbolic ideal tetrahedra, and prove the Ptolemy equation of \refthm{main_thm_3}. Take an ideal tetrahedron with vertices labelled $0,1,2,3$, and a spin decoration $(H_i, W_i)$ on each ideal vertex $i$. We must show that the complex lambda lengths $\lambda_{ij}$ from $(H_i, W_i)$ to $(H_j, W_j)$ satisfy
\begin{equation}
\label{Eqn:Ptolemy}
\lambda_{01} \lambda_{23} + \lambda_{03} \lambda_{12} = \lambda_{02} \lambda_{13}.
\end{equation}

\begin{proof}[Proof of \refthm{main_thm_3}]
Let $\kappa_i \in \C^2$ be the spinor corresponding to $(H_i, W_i)$. Let $M$ be the $2 \times 4$ complex matrices whose $j$'th column is $\kappa_j$, and let $M_{ij}$ be the $2 \times 2$ submatrix whose columns are $\kappa_i$ and $\kappa_j$ in order. Then $\det M_{ij} = \{ \kappa_i, \kappa_j \} = \lambda_{ij}$, so the claimed equation becomes
\[
\det M_{01} \det M_{23} + \det M_{03} \det M_{12} = \det M_{02} \det M_{13},
\]
which is a well known Pl\"{u}cker relation.
\end{proof}

Note that if we multiply any one of the spinors, say $\kappa_i$ corresponding to $(H_i, W_i)$, by a complex scalar $c$, each term of the Ptolemy equation \refeqn{Ptolemy} involving index $i$ is also scaled by $c$. For instance if we multiply $\kappa_1$ by $c$ then $\lambda_{01}, \lambda_{12}, \lambda_{13}$ are all multiplied by $c$. In some sense then the choice of decorated horosphere at each vertex is a choice of gauge. The equation is in a certain sense, just the usual equation
\begin{equation}
\label{Eqn:shape_equation}
z + z'^{-1} = 1,
\end{equation}
relating shape parameters for a hyperbolic ideal tetrahedron $\Delta$, as we see next. By the \emph{shape parameter} $z_e$ of $\Delta$ along an edge $e$, we mean the complex number $z$ such that, if we place the two endpoints of $e$ at $0$ and $\infty$ in $\partial \U^3 \cong \C \cup \{\infty\}$, and place the remaining two ideal vertices at $1$ and a point with positive imaginary part, then the final vertex lies at $z$. By this definition, opposite edges of $\Delta$ have the same shape parameter, and the three pairs of shape parameters can be denoted $z,z',z''$ such that $z' = \frac{1}{1-z}$ and $z'' = \frac{z-1}{z}$. In particular, \refeqn{shape_equation} holds, and continues to hold if we cyclically permute $(z,z',z'') \mapsto (z',z'',z)$.

\begin{prop}
\label{Prop:shape_parameters}
Numbering the ideal vertices of $\Delta$ by $0,1,2,3$ as in \reffig{Tetrvertexlabels}, let the shape parameter of edge $ij$ by $z_{ij}$. Choose a spin-decorated horosphere $(H_i, W_i)$ at ideal vertex $i$ and let $\lambda_{ij}$ be the complex lambda length from $(H_i, W_i)$ to $(H_j, W_j)$. Then
\begin{equation}
\label{Eqn:shape_parameters}
z_{01} = z_{23} = \frac{\lambda_{02} \lambda_{13}}{\lambda_{03} \lambda_{12}}, \quad
z_{02} = z_{13} = -\frac{\lambda_{03} \lambda_{12}}{\lambda_{01} \lambda_{23}}, \quad
z_{03} = z_{12} = \frac{\lambda_{01} \lambda_{23}}{\lambda_{02} \lambda_{13}}.
\end{equation}
\end{prop}

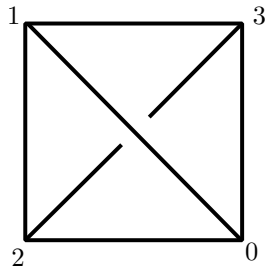
\begin{figure}[h]
\begin{center}
  %% Creator: Inkscape 1.2 (dc2aedaf03, 2022-05-15), www.inkscape.org
%% Accompanies image file 'TetrahedronLabels.pdf' (pdf, eps, ps)
\begingroup%
  \makeatletter%
  \providecommand\color[2][]{%
    \errmessage{(Inkscape) Color is used for the text in Inkscape, but the package 'color.sty' is not loaded}%
    \renewcommand\color[2][]{}%
  }%
  \providecommand\transparent[1]{%
    \errmessage{(Inkscape) Transparency is used (non-zero) for the text in Inkscape, but the package 'transparent.sty' is not loaded}%
    \renewcommand\transparent[1]{}%
  }%
  \providecommand\rotatebox[2]{#2}%
  \newcommand*\fsize{\dimexpr\f@size pt\relax}%
  \newcommand*\lineheight[1]{\fontsize{\fsize}{#1\fsize}\selectfont}%
  \ifx\svgwidth\undefined%
    \setlength{\unitlength}{95.56947441bp}%
    \ifx\svgscale\undefined%
      \relax%
    \else%
      \setlength{\unitlength}{\unitlength * \real{\svgscale}}%
    \fi%
  \else%
    \setlength{\unitlength}{\svgwidth}%
  \fi%
  \global\let\svgwidth\undefined%
  \global\let\svgscale\undefined%
  \makeatother%
  \begin{picture}(1,1.02439286)%
    \lineheight{1}%
    \setlength\tabcolsep{0pt}%
    \put(0,0){\includegraphics[width=\unitlength,page=1]{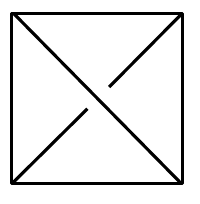}}%
    \put(0.92647168,0.02391656){\color[rgb]{0,0,0}\makebox(0,0)[lt]{\lineheight{0}\smash{\begin{tabular}[t]{l}0\end{tabular}}}}%
    \put(-0.01226202,0.95368135){\color[rgb]{0,0,0}\makebox(0,0)[lt]{\lineheight{0}\smash{\begin{tabular}[t]{l}1\end{tabular}}}}%
    \put(0.00567554,0){\color[rgb]{0,0,0}\makebox(0,0)[lt]{\lineheight{0}\smash{\begin{tabular}[t]{l}2\end{tabular}}}}%
    \put(0.95636764,0.95368186){\color[rgb]{0,0,0}\makebox(0,0)[lt]{\lineheight{0}\smash{\begin{tabular}[t]{l}3\end{tabular}}}}%
  \end{picture}%
\endgroup%

  \caption{Tetrahedron with vertices labeled $0$, $1$, $2$, $3$.}
  \label{Fig:Tetrvertexlabels}
\end{center}
\end{figure}

\begin{proof}
If we move a spin-decorated tetrahedron by a spin isometry, all shape parameters and complex lambda lengths remain invariant. Noting the orientation of \reffig{Tetrvertexlabels}, we may place the ideal vertices $0,1,2,3$ respectively at $0,\infty,z,1 \in \partial \hyp^3$ respectively, so $z=z_{01}$. With this arrangement then $z=z_{01}=z_{23}$, $z' = z_{02} = z_{13}$ and $z'' = z_{03}=z_{12}$. If we multiply a spinor $\kappa_i$ corresponding to $(H_i, W_i)$ by a complex scalar $c$, the homogeneous expressions in lambda lengths in \refeqn{shape_parameters} are invariant. Thus it suffices to prove the claim for any single choice of spin decoration, or spinor, at each vertex. Take spinors $\kappa_0 = (0,1)$, $\kappa_1 = (1,0)$, $\kappa_2 = (z,1)$, $\kappa_3 = (1,1)$. By \refthm{main_thm_2} then we calculate all complex lambda lengths as
\[
\lambda_{01} = -1, \quad
\lambda_{02} = -z, \quad
\lambda_{03} = -1, \quad
\lambda_{12} = 1, \quad
\lambda_{13} = 1, \quad
\lambda_{23} = z-1.
\]
This immediately gives the first equation.  Permuting labels $(0,1,2,3) \mapsto (0,2,3,1)$ on $\Delta$, similarly permuting indices on each $\lambda$, and using the antisymmetry of $\lambda$, then gives the subsequent two equations.
\end{proof}

When the ideal tetrahedron $\Delta$ is degenerate and lies in a plane, by an isometry we may place $\Delta$ on the upper half plane model $\U^2$ of the hyperbolic plane with $\partial \U^2 = \R \cup \{\infty\}$, i.e. the upper half plane model inside the upper half space model. All vertices at infinity then lie in $\R \cup \{\infty\}$, and we may choose all spin directions to point perpendicular to $\U^2$ in the following sense. Note that every horocycle $H$ centred at $z \in \R \cup  \{ \infty \}$ extends to a unique horosphere $H'$ in $\U^3$, also centred at $z$.
\begin{defn}
\label{Def:planar_spin_decoration}
Let $H$ be a horocycle in $\U^2 \subset \U^3$. A \emph{planar spin decoration} on $H$ is a spin decoration $(W^{in}, W^{out})$ on $H'$ such that $W^{in}, W^{out}$ project to frames specified by $i$.
\end{defn}
``Specified" here means north-pole specified, if $H'$ is a sphere in the upper half space model. 

\begin{lem}
\label{Lem:real_spinors_planar_decorations}
A spinor yields a planar spin decoration on a horocycle if and only if it is real.
\end{lem}

\begin{proof}
The spin-decorated horosphere of $\kappa = (\xi, \eta)$ will be a planar spin decoration on a horocycle if and only if the centre $\xi/\eta \in \R \cup \{\infty\}$ and the decoration direction $i/\eta^2$ (if $\eta \neq 0$) or $i \xi^2$ (if $\eta = 0$) is a positive real multiple of $i$. This happens precisely when $\xi, \eta$ are real.
\end{proof}

Thus, we can reduce to two dimensions by considering \emph{real} spinors, i.e. those in $\R_*^2 = \R^2 \setminus \{ (0,0) \}$. Then all complex distances $d_{ij}$ between horospheres are 
of the form $\rho + i \theta$ where $\theta \equiv 0$ mod $2\pi$,
 so the $\lambda_{ij} = \exp(d_{ij}/2)$ are 
real. Taking absolute values, we obtain
Penner's real lambda lengths between horocycles from \cite{Penner87, Penner12_book}.

A horocycle in $\U^2$ has two planar spin decorations, corresponding to the two spin lifts of frames specified by the $i$ direction. These two planar spin decorations correspond to two real spinors, which are negatives of each other.

\section{Cluster algebra applications}
\label{Sec:cluster_applications}

We now develop the notions required to prove \refthm{T_Gr}.

For reasons that will become apparent, we will consider $\hyp^2$ via the upper half plane model $\U^2$, but to have the orientation induced by the normal vector in the $i$ direction in the upper half space model; this is the opposite to the usual orientation. Then the boundary circle $\partial \U^2 \cong \R \cup \{\infty\}$ obtains an orientation in the \emph{negative} real direction.
\begin{defn}
An \emph{ideal $d$-gon} is a collection of distinct points $z_1, \ldots, z_d$ in $\partial \U^2$, labelled in order around the oriented boundary $\partial \U^2$. A \emph{decoration} on an ideal $d$-gon is a choice of horocycle at each $z_i$.
\end{defn}
We can join the points $z_1, \ldots, z_d$ (and back to $z_1$) successively by geodesics to form an ideal $d$-gon in the usual sense, but for our purposes we do not need the edges; for us a $d$-gon is just the sequence of vertices, we just need the $z_i$. Given the orientation on $\partial \U^2$, this means that the $z_i \in \R \cup \{\infty\}$ satisfy either
\begin{equation}
\label{Eqn:cyclic_ordering_all_real}
z_k > z_{k+1} > \cdots > z_d > z_1 > z_2 > \cdots > z_{k-1}
\end{equation}
or
\begin{equation}
\label{Eqn:cyclic_ordering_including_infinity}
 z_k = \infty \quad \text{and} \quad z_{k+1} > \cdots > z_d > z_1 > \cdots > z_{k-1}
\end{equation}
for some $k$.

In 3 dimensions, we lose the ordering on vertices of an ideal $d$-gon, and instead use the following weaker notion. Again, our definition is just a sequence of ideal vertices, although we can imagine joining them by geodesics.
\begin{defn}
\label{Def:skew_d-gon}
A \emph{skew ideal $d$-gon} is a collection of distinct points $z_1, \ldots, z_d \in \partial \U^3$. A \emph{spin decoration} on a skew ideal $d$-gon is a choice of spin-decorated horosphere centred at each $z_i$.
\end{defn}

We now use a special case of Penner's definition in \cite{Penner87} in 2 dimensions, and generalise it naturally to 3 dimensions.
\begin{defn} \
\label{Def:Teichmuller}
\begin{enumerate}
\item
The \emph{decorated Teichm\"{u}ller space} $\widetilde{T}(d)$ of ideal $d$-gons is the space of all decorated ideal $d$-gons, up to orientation-preserving isometries of $\U^2$.
\item
The \emph{decorated Teichm\"{u}ller space} $\widetilde{T}^3 (d)$ of skew ideal $d$-gons is the space of all spin-decorated skew ideal $d$-gons, up to spin isometries of $\U^3$.
\end{enumerate}
\end{defn}
In other words, the orientation-preserving isometry group $PSL(2,\R)$ acts on the space of decorated ideal $d$-gons, and $\widetilde{T}(d)$ is the quotient. Similarly, the spin isometry group $SL(2,\C)$ acts on the space of spin-decorated skew ideal $d$-gons, and $\widetilde{T}^3 (d)$ is the quotient.

In the 2-dimensional case, with real spinors, the following statements demonstrate that a notion of total positivity has nice hyperbolic-geometric consequences. Similar ideas also appear in the physics literature, e.g. \cite{ABCGPT16}.
\begin{defn}
A collection of spinors $\kappa_1, \ldots, \kappa_n$ is \emph{totally positive} if they are all real, and for all $i<j$ we have $\{\kappa_i, \kappa_j\} > 0$.
\end{defn}
Note that the totally positive condition implies that the $\kappa_i$ are all linearly independent, so the corresponding horospheres are all centred at distinct points $z_i \in \R \cup \{\infty\}$; and by antisymmetry, when $i>j$ we have $\{\kappa_i, \kappa_j\}<0$.

\begin{lem}
\label{Lem:spin-coherent_ideal_d-gons}
Let $d \geq 3$. If $\kappa_1, \ldots, \kappa_d$ are totally positive then the centres $z_i$ of the corresponding horospheres form an ideal $d$-gon. The planar spin-decorated horospheres of $\kappa_1, \ldots, \kappa_d$ yield a map
\[
\left\{ \text{Totally positive $d$-tuples of spinors} \right\}
\To
\left\{ \text{Decorated ideal $d$-gons} \right\} 
\]
which is surjective and 2-1, with the preimage of each ideal $d$-gon of the form $\pm (\kappa_1, \ldots, \kappa_d)$.
\end{lem}
In other words, the totally positive condition forces the $z_i$ to be in order around $\partial \U^2$, and we obtain a decorated ideal $d$-gon. Conversely, any decorated ideal $d$-gon is realised by precisely two $d$-tuples of totally positive real spinors, which are negatives of each other.

\begin{proof}
Letting $\kappa_i = (\xi_i, \eta_i)$ be totally positive we have
\begin{equation}
\label{Eqn:z_difference_eta_sign}
z_i - z_j = \frac{\xi_i}{\eta_i} - \frac{\xi_j}{\eta_j} = \frac{\{\kappa_i, \kappa_j\}}{\eta_i \eta_j}.
\end{equation}
Supposing $i<j$ then, we have $\{\kappa_i, \kappa_j\}>0$, so $\eta_i$ and $\eta_j$ have the same sign when $z_i > z_j$,  and $\eta_i$ and $\eta_j$ have opposite signs when $z_i < z_j$.

If $z_1, z_2, z_3$ are real and satisfy $z_1 < z_2 < z_3$, or $z_2 < z_3 < z_1$, or $z_3 < z_1 < z_2$, then we obtain a contradiction. We show this in the case $z_1 < z_2 < z_3$; the other cases are similar. From $z_1< z_2$, $\eta_1$ and $\eta_2$ have opposite signs. From $z_2 < z_3$, $\eta_2$ and $\eta_3$ have opposite signs. Thus $\eta_1$ and $\eta_3$ have the same sign. But $\eta_1$ and $\eta_3$ have opposite signs since $z_1 < z_3$, a contradiction.

This argument applies not just to $z_1, z_2, z_3$ but to any $z_i, z_j, z_k$ such that $i<j<k$. These are precisely the cases in which $z_i, z_j, z_k$ are not in order around $\partial \U^2$. Thus every triple of real numbers among the $z_i$ is in order around $\partial \U^2$. 

If all $z_i$ are real, then considering the triples $(z_1, z_2, z_3)$, $(z_1, z_3, z_4)$, $\ldots$, $(z_1, z_{d-1}, z_d)$ shows that all $z_i$ are in order around $\partial \U^2$, satisfying \refeqn{cyclic_ordering_all_real}.

Suppose some $z_k = \infty$, so $z_k = (\xi_k, 0)$. We then have $\{\kappa_i, \kappa_k\} = -\eta_i \xi_k$. For $i<k$ then $\eta_i \xi_k$ is positive, so $\eta_i$ has the same sign as $\xi_k$. Similarly, for $i>k$, $\eta_i$ has opposite sign to $\xi_k$. Thus if $i<k<j$ then $\eta_i, \eta_j$ have opposite signs, so from \refeqn{z_difference_eta_sign} then $z_i < z_j$. Applying the reasoning of the previous paragraph, it follows that \refeqn{cyclic_ordering_including_infinity} is satisfied.

Thus the $z_i$ are in order around $\partial \U^2$ and form an ideal $d$-gon, and by \reflem{real_spinors_planar_decorations} each $\kappa_i$ yields a planar spin-decorated horosphere at $z_i$, hence a horocycle decoration in $\partial \U^2$.

Conversely, suppose the $z_i$ with horocycles $H_i$ form a decorated ideal $d$-gon in $\U^2$. Each $z_i$ has two planar spin decorations, given by two real spinors of the form $\pm \kappa_i$. Choosing a sign for $\kappa_1$, the requirement that each $\{\kappa_1, \kappa_j\} > 0$ forces a choice for each other $\kappa_j$. This yields two possible $d$-tuples of real spinors describing the decorated ideal $d$-gon; we will show they are both totally positive.

Suppose all $z_i$ are real, satisfying \refeqn{cyclic_ordering_all_real} for some $k$. Then by \refeqn{z_difference_eta_sign} $\eta_i$ has the same sign as $\eta_1$ for $2 \leq i \leq k-1$, and $\eta_i$ has a different sign from $\eta_1$ for $k+1 \leq i \leq d$. It follows that, for $i<j$, $\eta_i$ and $\eta_j$ have the same sign precisely when $z_i < z_j$, so by \refeqn{cyclic_ordering_all_real} $\{\kappa_i, \kappa_j \} > 0$ for all $i<j$, and the $\kappa_i$ are totally positive.
\end{proof}

We can then give a description of $\widetilde{T}(d)$ in terms of totally positive spinors. The action of $SL(2,\R)$ on real spinors extends to an action on $d$-tuples as $A.(\kappa_1, \ldots, \kappa_d) = (A.\kappa_1, \ldots, A.\kappa_d)$. We then obtain the following.
\begin{prop}
\label{Prop:Td_totally_positive_spinors}
Let $d \geq 3$. The $SL(2,\R)$-orbits of totally positive $d$-tuples of real spinors are  naturally bijective with $\widetilde{T}(d)$:
\begin{equation}
\label{Eqn:spin-coherent_tuples_ideal_d-gons}
\frac{ \left\{ \text{Totally positive  $d$-tuples of real spinors} \right\} }{ SL(2,\R) }
\stackrel{\cong}{\To}
\frac{ \left\{ \text{Decorated ideal $d$-gons} \right\} }{ PSL(2,\R) } = \widetilde{T}(d).
\end{equation}
\end{prop}

\begin{proof}
We first show the map of \reflem{spin-coherent_ideal_d-gons}, sending totally positive $d$-tuples to decorated ideal $d$-gons, descends to a map as in \refeqn{spin-coherent_tuples_ideal_d-gons}. If two totally positive $d$-tuples are related by the action of $SL(2,\R)$, then by equivariance of the action of $SL(2,\R) \subset SL(2,\C)$, the resulting spin-decorated horospheres are related by a spin isometry of $\U^2 \subset \U^3$, and dropping spin structures, the underlying decorated ideal $d$-gons are related by an isometry in $PSL(2,\R)$.

Thus the map of \refeqn{spin-coherent_tuples_ideal_d-gons} exists. It is also surjective since the map of \reflem{spin-coherent_ideal_d-gons} is. To see that it is injective, note the map of \reflem{spin-coherent_ideal_d-gons} is 2--1, with the two preimages of a given decorated $d$-gon being negatives of each other. These two preimages  are related by the action of the negative identity in $SL(2,\R)$, giving a unique preimage.
\end{proof}

We now define the Grassmannians we need. For background and context on positive Grassmannians, see e.g. \cite{Baur21, Lusztig94, Postnikov06, Williams21, Williams14}. Recall that the Grassmannian $\Gr_\mathbb{F}(k,n)$ over a field $\mathbb{F}$ is the space of all $k$-planes in $\mathbb{F}^n$. It can be realised as the quotient of $\Mat^{k}_\mathbb{F} (k,n)$, the space of all $k \times n$ matrices over $\mathbb{F}$ of rank $k$, by the left action of $GL(k,\mathbb{F})$. A matrix represents the $k$-plane spanned by its rows. The $k \times k$ minors of a matrix yield $\binom{n}{k}$ projective coordinates on $\Gr_\mathbb{F}(k,n)$ called \emph{Pl\"{u}cker coordinates}. We only consider $k=2$ and $\mathbb{F} = \R$ or $\C$.
\begin{defn}
Let $\Mat_\R^+ (2,d)$ denote the space of all $2 \times d$ real matrices with all Pl\"{u}cker coordinates positive. The \emph{positive Grassmannian} $\Gr^+(2,d)$ is the quotient of $\Mat_\R^+ (2,d)$ by the left action of $GL^+(2,\R)$. The \emph{positive affine Grassmannian} $X^+ (d)$ is the affine cone on $\Gr^+(2,d)$.
\end{defn}
The Pl\"{u}cker coordinates on a Grassmannian are only defined up to an overall factor, but they provide bona fide coordinates on the affine cone.

The affine cone on the Grassmannian $\Gr_\R (2,d)$, as in \cite[example 12.6]{Fomin_Zelevinsky03}, can be identified with the nonzero decomposable elements of the second exterior power $\Lambda^2 ( \R^d )$. The plane spanned by two rows $R_1, R_2$ in a $2 \times d$ matrix is represented by $R_1 \wedge R_2$, and the action of $A \in GL(2,\R)$ is by
\[
A.(R_1 \wedge R_2) = (a R_1 + b R_2) \wedge (c R_1 + d R_2) =  ( \det A  ) \, R_1 \wedge R_2
\quad \text{where} \quad
A = \begin{bmatrix} a & b \\ c & d \end{bmatrix}.
\]
Taking the quotient by $GL(2,\R)$ thus identifies matrices whose corresponding decomposable elements of $\Lambda^2 (\R^d)$ represent the same $2$-plane in $\R^d$. Taking the quotient by $SL(2,\R)$ identifies matrices whose corresponding elements of $\Lambda^2 (\R^d)$ are equal, and thus the affine cone on $\Gr_\R(2,d)$ is the quotient of $\Mat^2_\R (2,d)$ by the left action of $SL(2,\R)$. Restricting to matrices in $\Mat_\R^+ (2,d)$, taking the quotient by $GL^+(2,\R)$ again identifies matrices whose corresponding decomposable elements of $\Lambda^2 (\R^d)$ which represent the same $2$-plane, and taking the quotient by $SL(2,\R)$ identifies matrices whose corresponding elements of $\Lambda^2 (\R^d)$ are equal. Thus $X^+ (d)$ is the quotient of $\Mat_\R^+ (2,d)$ by the left action of $SL(2,\R)$.

\begin{proof}[Proof of \refthm{T_Gr}(i)]
We now have, by \refprop{Td_totally_positive_spinors} and the above discussion
\[
\widetilde{T}(n) \cong \frac{\{ \text{Totally positive $d$-tuples of spinors} \}}{SL(2,\R)}
\quad \text{and} \quad
X^+ (n) = \frac{\Mat_\R^+ (2,d) }{SL(2,\R)}.
\]
Placing a $d$-tuple of real spinors $(\kappa_1, \ldots, \kappa_d)$ as the columns of a $2 \times d$ matrix, the totally positive condition is that $\{\kappa_i, \kappa_j\} > 0$ for $i<j$. Each such $\{\kappa_i, \kappa_j\}$ is then none other than the determinant of the $2 \times 2$ minor formed by columns $i$ and $j$, i.e. the Pl\"{u}cker coordinate $p_{ij}$, so we precisely obtain the matrices in $\Mat_\R^+ (2,d)$. The actions of $SL(2,\R)$ on totally positive $d$-tuples and $\Mat_\R^+ (2,d)$ are identical, so we obtain an identification $\widetilde{T}(n) \To X^+ (n)$. By \refthm{main_thm_2} each (complex) lambda length $\lambda_{ij}$ on $\widetilde{T}(n)$ is equal to $\{\kappa_i, \kappa_j\}$, which we have seen is equal to the Pl\"{u}cker coordinate $p_{ij}$ on $X^+ (n)$.
\end{proof}

In a similar fashion over $\C$, we can consider the subvariety of the Grassmannian where all Pl\"{u}cker coordinates are nonzero.
\begin{defn}
Let $\Mat_\C^* (2,d)$ denote the space of all $2 \times d$ complex matrices with all Pl\"{u}cker coordinates nonzero. The \emph{nonzero Grassmannian} $\Gr_\C^* (2,d)$ is the quotient of $\Mat_\C^* (2,d)$ by the left action of $GL(2,\C)$. The \emph{nonzero affine Grassmannian} $X^* (d)$ is the affine cone on $\Gr_\C^* (2,d)$.
\end{defn}
Again the affine cone on $\Gr_\C (2,d)$ can be identified with nonzero decomposable elements in $\Lambda^2 (\C^d)$, and taking the quotient by $SL(2,\C)$ identifies precisely those matrices whose corresponding elements of $\Lambda^2 (\C^d)$ are equal. Thus $X^* (d)$ is the quotient of $\Mat_\C^* (2,d)$ by the left action of $SL(2,\C)$.

\begin{proof}[Proof of \refthm{T_Gr}(ii)]
In a spin-decorated skew ideal $d$-gon, at each ideal vertex $z_i$ we have a spin-decorated horosphere corresponding to a spinor $\kappa_i$. The fact that all $z_i$ are distinct (\refdef{skew_d-gon}) implies that for all $i \neq j$ we have $\{\kappa_i, \kappa_j\} \neq 0$. By \refdef{Teichmuller}, $\widetilde{T}^3 (d)$ is the space of all spin-decorated skew ideal $d$-gons, up to spin isometries, so
\[
\widetilde{T}^3 (d) = \frac{ \{ \text{$d$-tuples of spinors with $\{\kappa_i, \kappa_j\} \neq 0$ for $i \neq j$ } \} }{ SL(2,\C) }
\quad \text{and} \quad
X^* (d) = \frac{ \Mat_\C^* (2,d) }{ SL(2,\C) }.
\] 
Again, putting the $d$ spinors as the columns of a matrix and noting that the $SL(2,\C)$ actions are identical gives an identification $\widetilde{T}^3 (d) \To X^* (d)$, and each complex lambda length $\lambda_{ij} = \{\kappa_i, \kappa_j\}$ is equal to the corresponding Pl\"{u}cker coordinate $p_{ij}$.
\end{proof}

\section*{Declarations}

{\flushleft \textbf{Conflict of Interests / Competing Interests statement.} }
The author has no conflicting or competing financial or non-financial interests directly or indirectly related to this work.

{\flushleft \textbf{Data availability statement.} }
This work has no associated data.

\small

\bibliography{spinref}
\bibliographystyle{amsplain}

\end{document}